\documentclass[a4paper,12pt]{article}
\usepackage{amsfonts}
\usepackage{amssymb}
\usepackage{latexsym}
\usepackage{graphicx, booktabs}
\usepackage{mathrsfs}
\usepackage{enumitem}
\usepackage{amsmath}
\usepackage{graphicx}
\usepackage{bbm}
\usepackage{xcolor}
\usepackage{mathtools}
\usepackage{adjustbox}
\usepackage{dirtytalk}
\usepackage{array}
\usepackage{multirow}

\RequirePackage{amsthm,amsfonts,amssymb}
\RequirePackage[colorlinks,citecolor=blue,urlcolor=blue]{hyperref}
\newcommand{\rmd}{\,\mathrm{d}}
\newcommand{\expectation}[1]{\mathbb{E}\left[{#1}\right]}
\newcommand{\converginP}{\overset{\mathbb{P}}{\rightarrow}}
\newcommand{\converginD}{\overset{\mathcal{D}}{\rightarrow}}
\newcommand{\R}{\mathbb{R}}
\newcommand{\N}{\mathbb{N}}
\newcommand{\PP}{\mathbb{P}}
\newcommand{\F}{\mathcal{F}}
\newcommand{\E}{\mathbb{E}}
\newcommand{\tr}{^\top}
\newcommand{\defeq}{\coloneqq}
\newcommand{\1}{\mathbbm{1}}

\theoremstyle{plain}

\newtheorem{theorem}{Theorem}[section]
\newtheorem{lemma}[theorem]{Lemma}
\newtheorem{proposition}[theorem]{Proposition}

\theoremstyle{remark}
\newtheorem{remark}[theorem]{Remark}
\newtheorem{definition}[theorem]{Definition}

\begin{document}
\title{Likelihood inference of the non-stationary Hawkes process with non-exponential kernel}
\author{Tsz-Kit Jeffrey Kwan, Feng Chen, and William Dunsmuir\\School of Mathematics and Statistics \\UNSW Sydney}
\maketitle
\begin{abstract}
  The Hawkes process is a popular point process model for event sequences that exhibit temporal clustering. The
  intensity process of a Hawkes process consists of two components, the baseline intensity and the accumulated
  excitation effect due to past events, with the latter specified via an excitation kernel. The classical Hawkes process
  assumes a constant baseline intensity and an exponential excitation kernel. This results in an intensity process that
  is Markovian, a fact that has been used extensively to establish the strong consistency and asymtpotic normality of
  maximum likelihood estimators or similar. However, these assumptions can be overly restrictive and unrealistic for
  modelling the many applications which require the baseline intensity to vary with time and the excitation kernel to
  have non-exponential decay. However, asymptotic properties of maximum likelihood inference for the parameters
  specifying the baseline intensity and the self-exciting decay under this setup is substantially more difficult since
  the resulting intensity process is non-Markovian. To overcome this challenge, we develop an approximation procedure to
  show the intensity process is asymptotically ergodic in a suitably defined sense. This allows for the identification
  of an ergodic limit to the likelihood function and its derivatives, as required for obtaining large sample inference
  under minimal regularity conditions.
\end{abstract}

\section{Introduction}\label{sec:Introduction}

Hawkes \cite{hawkes1971point,hawkes1971spectra,hawkes1972spectra,hawkes1974cluster} introduced a class of point processes with the essential
capability of modelling phenomena that exhibit temporal clustering, whereby the occurrence of an event increases the likelihood of additional
events to occur in the future. Such processes are termed self-exciting point processes and are often denoted by $N$. Given infinite history, the complete intensity function of $N_t$ is given by
\begin{equation} \label{eq:complete_intensity}
	\lambda_t = \nu + \int_{-\infty}^{t-}g(t-s)\rmd N_s
\end{equation}
where $\nu$ is known as the baseline intensity and $g$ the excitation kernel. 

Throughout, we write $g(t) = \eta \tilde{g}(t)$ where $\eta=\int_0^\infty g(t)\rmd t$ is called the branching ratio, and the normalized excitation kernel $\tilde{g}(t)=g(t)/\eta$ is a probability density function on $[0,\infty)$. To ensure the process does not explode over time, we also assume $\eta <1$.
 The second term of $\lambda(t)$ is responsible for the self-exciting nature of the point process. Traditionally, the baseline intensity is assumed to be a constant
and the excitation kernel is assumed to take the exponential form, $\tilde g(t) = \beta e^{-\beta t}$, $\beta > 0$. As a consequence of
the memorylessness of the exponential distribution, the bivariate process $(N_t, \lambda_t)$, $t\geq 0$ is Markovian, which simplifies likelihood calculations significantly \cite{oakes1975markovian}. In addition, this also allows for asymptotic procedures from the large literature of Markov
processes to be applied to the classical Hawkes process. For example, \cite{abergel2015Long} modelled the arrivals of trades in a limit
order book as a multivariate Hawkes process with an exponential kernel. They make use of the ergodic property of Markov processes to derive the diffusive behaviour of the price at large time scales from the microscopic properties of the model. In \cite{kwan2022alternative}, the authors approximated a
Hawkes process with time-varying baseline intensity and exponential kernel by a sequence of Markovian Hawkes processes and identified a deterministic
limit to the log-likelihood function. Subsequently, they established consistency and asymptotic normality of the maximum likelihood estimator (MLE).

While the classical assumptions imposed on the baseline intensity and excitation kernel are convenient for studying the
probabilistic properties of the process, they also substantially limit the applicability of the Hawkes process to many
real-life situations. For instance, in modelling intra-day trading data, it is well observed that trading intensity is
typically higher following market open and prior to market close and lower in the middle of the day
\cite{engle1998autoregressive}, thus suggesting a time-varying baseline intensity $\nu(t)$ may provide a more
appropriate fit \cite{chen2013inference, kwan2023asymptotic, stindl2019extreme}. Furthermore, the exponential kernel has
a maximum at $t=0$, suggesting the peak influence of a past event occurs immediately after it occurs. Thus, a Hawkes
process with an exponential kernel has limited usage in epidemiology, since it fails to reflect the latent period
between when an individual is infected and when they become infectious \cite{lauer2020incubation}. In such situations, a
Gamma or Weibull kernel with shape parameter $\alpha > 1$ may be more appropriate \cite{escobar2020hawkes}. In
high-frequency financial trading (for example, when modelling the limit order book) it has often been observed that the
decay kernel is not exponential but rather exhibits heavy tail decay, which is more appropriately modelled by a
heavy-tailed kernel such as the power law~\cite{bacry2015hawkes, gould2013limit}.

When the excitation kernel is non-exponential, the resulting Hawkes process is non-Markovian in general, which renders
standard techniques from ergodic theory of Markov processes inapplicable. As a result, it is substantially more
difficult to derive ergodicity of the intensity and asymptotic results for the likelihood. In \cite{kwan2024ergodic},
the authors established ergodicity of the intensity process of a Hawkes process with a constant baseline intensity with
a non-exponential excitation kernel. From the Poisson cluster process representation of the Hawkes process, the authors
established the ergodicity of the intensity process from the mixing property of Poisson cluster process. Consequently, a
deterministic limit to the likelihood function can be identified and consistency of the MLE was shown. In this paper, we
consider Hawkes processes where not only is the excitation kernel not necessarily exponential, but the baseline
intensity can also be time-varying.  Due to the non-stationary nature of the resulting process, an infill asymptotics
regime is adopted to investigate the asymptotic properties of the process, in which the process is observed over a fixed
interval as the total number of events/points observed increases. This is done by considering a sequence of Hawkes
processes indexed by $n$, with the intensity of the process tending to infinity with $n$. With the goal of establishing
consistency and asymptotic normality of the MLE in mind, we first establish asymptotic ergodicity of the intensity
process and its partial derivatives with respect to the model parameters. Once the asymptotic ergodicity of the
intensity process is established, a deterministic limit to the likelihood function can be identified, from which
consistency and asymptotic normality of the MLE follow.

This article is organized as follows. In Section~\ref{sec:setup-notation}, we formally introduce the parametric Hawkes
process with a time-varying baseline intensity and a general kernel function.  Section~\Ref{sec:main_results} is
dedicated to establishing asymptotic ergodicity of the intensity process and its partial derivatives with respect to the
parameters up to order 2 (Theorem~\ref{thm:asympergodicity}). An outline of the proof can also be found in
Section~\Ref{sec:main_results} with the technical details relegated to
Appendix~\Ref{sec:ErgodicTheory-prep_Proofs}. Section~\Ref{sec:parametricinference} is devoted to some applications of
Theorem~\ref{thm:asympergodicity} to show consistency (Theorem~\ref{thm:consistency}) and asymptotic normality
(Theorem~\ref{thm:asym-norm}) of the MLE. The proof of Theorem~\ref{thm:consistency} relies on the identification of a
deterministic limit of the normalised log-likelihood function, which is made possible by
Theorem~\ref{thm:asympergodicity}, and an application of the M-estimator master theorem
\cite[Theorem~5.7]{van2000asymptotic}. Asymptotic ergodicity of the intensity process and its derivatives is used,
again, in the proof of Theorem~\ref{thm:asym-norm} to identify the variance structure of MLE. Examples of
non-exponential kernels that satisfy the assumed conditions for which consistency and asymptotic normality are achieved
can be found in Section~\ref{sec:exampl-numer-simul}. Simulation results to verify consistency and asymptotic normality
are also provided in the same section. We conclude the paper in
Section~\ref{sec:conclusion}. Appendix~\ref{sec:preperation-proofs} is devoted to the statements and proofs of some
boundedness results that are used in the proofs of
Theorems~\ref{thm:asympergodicity}--\ref{thm:asym-norm}. Appendix~\ref{sec:ErgodicTheory-prep_Proofs} contains the
technical details of the proof of the results in Section~\ref{sec:main_results}.

\section{Setup and notation}\label{sec:setup-notation}
In practice, only events starting from some fixed time origin are observable. Furthermore, as mentioned before, we
studied the process under infill asymptotic as mentioned above. For these reasons, let us consider a sequence of Hawkes
processes with a time-varying baseline intensity and a general excitation kernel $\tilde{N}^n$ on the interval $[0,1]$
for $n = 1, 2, \ldots$. For each $n$, denote by $\tilde{\mathcal{F}}^n = (\tilde{\mathcal{F}}^n_t)_{t \in [0,1]}$ the
natural filtration induced by $\tilde{N}$, with $\tilde{\mathcal{F}}^n_t = \sigma\{ \tilde{N}^n_s: 0 \le s \le t
\}$. The $\tilde{\mathcal{F}}^n$-intensity of $\tilde{N}^n_t$ is
\begin{equation}\label{eq:intensity1}
    \tilde{\lambda}^n_t \defeq n\nu(t) + \int_{0}^{t-} n g(n(t-s))\rmd \tilde{N}^n_s, \quad t \in [0,1],
\end{equation}
where $\nu(\cdot)$ is the (normalised) baseline intensity function, and $g(\cdot) \ge 0$ is the unnormalized excitation
kernel such that $\int_0^\infty g(t) \rmd t < 1$. Furthermore, we assume the function $\nu(\cdot) = \nu(\cdot ; \theta)$
and $g(\cdot) = g(\cdot; \theta)$ are known up to a finite-dimensional parameter
$\theta \in \Theta \subset \mathbb{R}^d$. The likelihood of $\theta$ relative to the sample path of $\tilde{N}^n$ from
$0$ to $1$, up to some multiplicative constant, takes the form
\begin{equation*}
    \tilde{l}^n(\theta) = \exp\left( \int_0^1 \log (\tilde{\lambda}^n){t, \theta} \rmd \tilde{N}^n_t - \int_0^1 \tilde{\lambda}^n_{t, \theta} \rmd t\right)
\end{equation*}
where $\tilde{\lambda}^n_{t, \theta}$ is as in \eqref{eq:intensity1} with $\nu(\cdot)$ and $g(\cdot)$ replaced by $\nu(\cdot; \theta)$ and $g(\cdot;\theta)$ respectively. The normalised log-likelihood function is given by
\begin{equation*}
    \tilde{\mathcal{L}}^n(\theta) = \frac{1}{n}\tilde{l}^n(\theta) = \frac{1}{n} \int_0^1 \log (\tilde{\lambda}^n_{t,\theta}) \rmd \tilde{N}^n_t - \frac{1}{n} \int_0^1 \tilde{\lambda}^n_{t, \theta} \rmd t.
\end{equation*}
The MLE $\hat{\theta}^n$ is computed by maximising $\tilde{\mathcal{L}}^n(\theta)$ with respect to $\theta$ over $\Theta$.

An expanding time domain allows for a more convenient way to study the asymptotic properties of the process \cite{kwan2022alternative}. For this reason, we apply a time change to stretch the time domain of the process from $[0,1]$ to $[0, n]$. Let $N^n_t \defeq \tilde{N}^n_{t/n}$, $\mathcal{F}^n_t \defeq \sigma \{ N^n_s: 0 \le s \le t\} =\tilde{\mathcal{F}}^n_{t/n}$, for $t \in [0,n]$, $\lambda^n_{t, \theta}$ the intensity of $N^n$ relative to $\mathcal{F}^n$ where $\mathcal{F}^n \defeq (\mathcal{F}^n_t)_{t \in [0,n]}$ so that 
\begin{equation}\label{eq:1}
  \lambda^n_t \defeq \nu\left(\frac{t}{n}\right) + \int_{0}^{t-}g(t-s)\rmd N^n_s, \quad t \in [0,n].
\end{equation}
The likelihood of $\theta$ can be written, up to some multiplicative constant, by
\begin{equation}
  l^n(\theta) = \exp \left\{\int_{0}^{n}\log(\lambda^n_{t,\theta}) \rmd {N}^n_t - \int_{0}^{n}{\lambda}_{t, \theta}^n\rmd t
    \right\}, 
\end{equation}
where ${\lambda}^n_{t, \theta}$ is as in \eqref{eq:1} with $\nu(\cdot)$ and $g(\cdot)$ replaced by $\nu(\cdot; \theta)$
and $g(\cdot; \theta)$ respectively. The normalised log-likelihood, score function and the negative Hessian matrix are
respectively

\begin{align}
  {\mathcal{L}}^n(\theta) & {} = \frac{1}{n}\log({l}(\theta)) = \frac{1}{n}\int_{0}^{n} \log({\lambda}^n_{t, \theta}) \rmd
                            {N}^n_t - \frac{1}{n}\int_{0}^{n}\lambda^n_{t, \theta} \rmd t,\label{LogLikelihood}\\
  {S}^n(\theta) & {} = \partial_\theta {\mathcal{L}}^n(\theta) = \frac{1}{n} \int_{0}^{n} \frac{\partial_\theta {\lambda}^n_{t,
                  \theta}}{{\lambda}^n_{t,\theta}} \rmd {N}^n_t - \frac{1}{n}\int_{0}^{n} \partial_\theta {\lambda}^n_{t, \theta} \rmd t,\label{ScoreFunction} \\
  {I}^n(\theta) & {} = - \partial_\theta^{\otimes 2} {\mathcal{L}}^n(\theta) =
                  -\frac{1}{n}\int_{0}^{n}\left(\frac{\partial_\theta^{\otimes 2} {\lambda}^n_{t,\theta}}{{\lambda}^n_{t,\theta}} -
                  \left( \frac{\partial_\theta^{\otimes 2} {\lambda}^n_{t,\theta}}{{\lambda}^n_{t, \theta}}\right)^{\otimes 2}
                  \right) \rmd {N}^n_t + \frac{1}{n}\int_{0}^{n}\partial_\theta^{\otimes 2} {\lambda}^n_{t, \theta} \rmd t,\label{InformationMatrix}
\end{align}
where $\partial_\theta = \frac{\partial}{\partial \theta}$, and $\partial_\theta^{\otimes 2}=\frac{\partial}{\partial \theta} \frac{\partial}{\partial
  \theta}^\top$. Note that $\mathcal{L}^n(\theta) = \tilde{\mathcal{L}}^n(\theta) - \log(n)\tilde{N}^n_1/n$ and therefore, the MLE $\hat{\theta}^{n}$ maximizes both $\tilde{\mathcal{L}}^n(\theta)$ and $\mathcal{L}^n(\theta)$ over the parameter space $\Theta$.
  
Fundamental to proving consistency and asymptotic normality of the MLE, istto provide an appropriate definition of ergodicity of the stochastic processes generating the observations. A suitable definition is \cite[Definition~3.1]{clinet2017statistical}.
\begin{definition}\label{def:ergodicity}
Let $(E, \mathcal{B}(E))$ be a Borel space and $X : \Omega \times \mathbb{R} \to E$ be a stochastic process. $X$ is $C$-ergodic for (or ergodic with
respect to) some family of functions $C: E \to \mathbb{R}$ if for any $\psi \in C$, there exists a mapping $\pi: C \to \mathbb{R}$ such that
\begin{equation}
\frac{1}{T}\int_{0}^{T}\psi (X_s) \rmd s \overset{\mathbb{P}}{\rightarrow} \pi(\psi)
\end{equation}
as $T \to \infty$. Here, $\pi(\psi)$ is also known as the ergodic limit of $\psi(X_s)$. 
\end{definition}

Next, let us introduce a few terms and notation. If $x$ is a real number, a vector or a matrix, $\vert x \vert_p = (\sum_i \vert x_i \vert^p)^{1/p}$
and $\vert \cdot \vert = \vert \cdot \vert_1$. For a vector $x \in \mathbb{R}^d$, $x^{\otimes 2}$ denotes the Kronecker product of $x$ and
$x^{\otimes 2} = x x^\intercal \in \mathbb{R}^{d \times d}$. If $X$ is a random variable in $\mathbb{R}^d$,
$\Vert X \Vert_p = \expectation{\sum_i\vert X_i \vert^p}^{1/p}$. If $a_n$ and $b_n$ are two sequences that depend on $n$, $a_n = O(b_n)$ implies
$\frac{a_n}{b_n} \rightarrow k$ as $n \rightarrow \infty$ for some finite constant $k$. For some Borel space $(E, \mathcal{B}(E))$, denote
$C_b(E, \mathbb{R})$ the set of functions $\psi:E \to \mathbb{R}$ such that $\psi$ is continuous and bounded, and $C_p(E,\mathbb{R})$ the set of
functions $\psi:E \to \mathbb{R}$ such that $\psi$ is continuous on $E$ and of polynomial growth of order $p$ in $E$. That is, for any $p >0$,
$\psi (x)= O(\vert x \vert^p)$ for all $\psi \in C_p(E, \mathbb{R})$.  Finally, we let $D_p(E,\R)$ denote the set of functions $\psi$ with continuous
first derivatives on $E$ such that $\psi \in C_p(E, \R)$ and $\vert \nabla\psi\vert \in C_{p-1}(E, \R)$.

\section{Asymptotic Ergodicity}\label{sec:main_results}
In this section, we investigate the asymptotic properties of the intensity process $\lambda^n_{t, \theta}$ and its partial derivatives, up to order 2,
with respect to $\theta$ for the purpose of establishing consistency and asymptotic normality of the MLE and let
$Y^n_{t, \theta, \theta^\ast} \defeq (\lambda^n_{t, \theta}, \lambda^n_{t, \theta^\ast}, \partial_\theta \lambda^n_{t, \theta},
\partial_\theta^{\otimes 2}\lambda^n_{t, \theta})$ taking values in $E_1=\mathbb{R}_+\times \mathbb{R}_+ \times \mathbb{R}^d
\times \mathbb{R}^{d \times d}$. Note that Definition~\ref{def:ergodicity} is not appropriate in the current context since the
object of interest, $Y^n_{t, \theta, \theta^\ast}$, is a sequence of stochastic processes indexed by $n$, rather than a single process as in that definition.  For this reason, we use the following more general definition when discussing ergodicity of a sequence of stochastic processes.

\begin{definition}\label{def:asymptotic_ergodicity}
  For any Borel space $(E, \mathcal{B}(E))$, let $X^n: \Omega \times [0,n] \to E$ be a sequence of stochastic processes. $X^n$ is asymptotically
  $C$-ergodic if for any function $\psi \in C: E \to \R$ , there exists a mapping $\pi : C \to \R$ such that
	\begin{equation*}
		\frac{1}{n} \int_0^n \psi(X^n_s) \rmd s \overset{\mathbb{P}}{\rightarrow} \pi(\psi)
	\end{equation*}
	as $n \to \infty$.
\end{definition}
We now give the regularity conditions under which we will establish asymptotic ergodicity of $Y^n_{t, \theta, \theta^\ast}$.

\begin{enumerate}[label= {[C1]} ]
	\item \label{cond:C1} $\nu$ and $g$ depend on the parameter $\theta = (\theta_1, \ldots, \theta_d)\tr$ that belongs to a parameter space $\Theta \subset \R^d$ where $\Theta$ is compact and contains a nonempty open ball in $\R^d$ which in turn contains $\theta^\ast$, the true parameter. Furthermore, for any $(x, \theta) \in [0,1] \times \Theta$, there exists $\underline{\nu}, \bar{\nu}, \underline{\eta}$ and $\bar{\eta}$ such that 
	\begin{align*}
		0 < \underline{\nu} \le \nu(x; \theta)  \le \bar{\nu} < \infty,\\
		0 < \underline{\eta} \le \eta \le \bar{\eta} < 1.
	\end{align*}
\end{enumerate}

\begin{enumerate}[label= {[C2]} ]
	\item \label{cond:C2} For $j \in \{0, 1, 2\}$, the family of functions $\{ \partial_\theta^{\otimes j} \nu(x; \theta)\}_{\theta \in \Theta}$ is Lipschitz equicontinuous. That is, there exists some constant $K \ge 0$ such that for any $j \in \{0, 1, 2\}$ and $\theta \in \Theta$, 
 \begin{equation*}
     \Vert \partial_\theta^{\otimes j} \nu(x_1; \theta) - \partial_\theta^{\otimes j}\nu(x_2;\theta) \Vert \le K\vert x_1 - x_2\vert.
 \end{equation*}
\end{enumerate}

\begin{enumerate}[label= {[C3]}($p$) ]
\item \label{cond:C3} For all $\theta \in \Theta$ and $j \in \{0,1,2\}$, and some $p\geq1$,
  $\vert \partial_\theta^{\otimes j} g(t; \theta) \vert^p$ is integrable with respect to $t$.
\end{enumerate}

Note that \ref{cond:C3} is formulated to depend on $p \ge 1$, and we have \hyperref[cond:C3]{[C3]($q$)} implies
\ref{cond:C3} for all $q \ge p\geq1$.

\begin{theorem}\label{thm:asympergodicity}
  For any $p \in \N_+$ and $\epsilon >0$, under \ref{cond:C1}, \ref{cond:C2}, and
  \hyperref[cond:C3]{[C3]($p + \epsilon$)}, $Y^n_{t, \theta, \theta^\ast}$ is asymptotically $D_p(E_1,
  \R)$-ergodic. That is, there exists a mapping $\pi: \Theta \times \Theta \times D_p(E_1, \R) \to \R$ such that for any
  $\psi \in D_p(E_1, \R)$,
  \begin{equation}\label{eq:asympergodicityeq1}
    \left\vert \frac{1}{n} \int_0^n \psi(Y^n_{t, \theta, \theta^\ast}) \rmd t- \pi(\theta, \theta^\ast; \psi) \right\vert \overset{\PP}{\rightarrow} 0, \quad n \rightarrow \infty.
  \end{equation}
  Furthermore, there exists a mapping $\dot{\pi}: [0, 1] \times \Theta \times \Theta \times D_p(E_1, \R) \to \R$ with
  $\dot{\pi}(x, \theta, \theta^\ast; \psi) = \E[\psi(\dot{Y}^x_{t, \theta, \theta^\ast})]$ which is continuous in $x$
  and $\pi(\theta, \theta^\ast; \psi) = \int_0^1 \dot{\pi}(x, \theta, \theta^\ast; \psi) \rmd x$.
\end{theorem}

The proof of Theorem~\ref{thm:asympergodicity} is established through an approximation method analogous to those adopted in the works of \cite{clinet2018statistical}, as well as
\cite{kwan2022alternative}. To facilitate an intuitive understanding of the proof of Theorem~\ref{thm:asympergodicity}, let us first delineate some notation, followed by an outline of the proof.

The idea of the proof is to split the time domain of the original process into $a_n$ number of blocks of equal width. Let $a_n$ be a monotone sequence of positive integers that tends to infinity such that $a_n = o(n)$. The time domain $[0,n]$ is split into $a_n$ number of equal-lengthed blocks with length
$b_n = n/a_n$, and denote the starting time of the $i$-th block by $\tau^n_{i-1} = (i-1)b_n$. Furthermore, write $\zeta^{i,n}: t \mapsto t + \tau^n_{i-1}$ as the time-shift operator for the $i$-th
block. Then, by considering the point process $N^n$ on $((i-1)b_n, ib_n]$, we can construct another point process $N^{i,n}$ defined on $[0, b_n]$ by
\begin{equation*}
	N^{i,n}_t \defeq N^n_{\zeta^{i,n}_t} - N^n_{\tau^n_{i-1}}.
\end{equation*}
$N^{i,n}$ counts the points of $N^n$ on $((i-1)b_n, i b_n]$ and is adapted to the filtration $\mathcal{F}^{i,n}_t \defeq \mathcal{F}^n_{\zeta^{i,n}_t}$. $N^{i,n}_t$ has stochastic intensity
$\lambda^{i,n}_{t, \theta^\ast}$ where for any
$\theta$, we have
\begin{align*}
	\lambda^{i,n}_{t, \theta} \defeq {} & \nu(\frac{\tau^n_{i-1} + t}{n}; \theta) + \int^{\tau^n_{i-1} + t -}_0 g(\tau^n_{i-1} + t - s; \theta) \rmd N^n_s \\
	= {} & \nu(\frac{\tau^n_{i-1} + t}{n}; \theta) + \int^{\tau^n_{i - 1} -}_0 g(\tau^n_{i-1} + t - s; \theta) \rmd N^n_s + \int_0^{t-} g(t-s; \theta) \rmd N^{i,n}_s,\quad t \in [0, b_n].
\end{align*}
We can now express the point process $N^n$ as the superposition of each $N^{i,n}$ as
\begin{equation}
	N^n_t = \sum^{\lfloor t/b_n \rfloor}_{i=1} N^{i,n}_{b_n} + N^{\lfloor t/b_n \rfloor + 1, n}_{t - \lfloor t / b_n \rfloor b_n}.
\end{equation}
By utilising the smoothness conditions imposed on the baseline intensity under \ref{cond:C2}, the baseline intensity of
each $N^{i,n}$ approaches some constant as $n$ tends to infinity. Consequently, each $N^{i,n}$ can be approximated by a
stationary Hawkes process with a constant baseline intensity, which, in turn, converges in probability to some limiting
process that is ergodic in the sense of Definition~\ref{def:ergodicity}. To explain further, we require the following
definitions.

Note that we can define another point process $N^{x,n}$ such that for any fixed $n \in \N_+$ and
$i \in \{1, \cdots, a_n\}$, there exists $x \in [0,1]$ such that $N^{x,n} \equiv N^{i,n}$.  For any $x$ and $n$, denote
another time-shift operator by $\zeta^{x,n}: t \mapsto t + nx$ and let
$\mathcal{F}^{x,n} \defeq (\mathcal{F}^{x,n}_t)_{t \in [0,b_n]}$ with
$\mathcal{F}^{x,n}_t \defeq \mathcal{F}^n_{\zeta ^{x,n}_t}$ be the filtration to which $N^{x,n}$ is adapted. $N^{x,n}_t$
has stochastic intensity $\lambda^{x,n}_{t, \theta^\ast}$ where for any $\theta$, we have
\begin{equation*}
  \lambda^{x,n}_{t, \theta} \defeq \nu(x + \frac{t}{n};\theta) + \int_0^{xn-} g(xn + t -s; \theta) \rmd N^n_s + \int_0^{t-} g(t-s; \theta) \rmd N^{x,n}_s, \quad t \in [0,n].
\end{equation*}
Next, let us define two auxiliary point processes $N^{x,n,c}$ and $\dot{N}^{x,n,c}$. Denote $\mathcal{F}^{x,n,c} \defeq (\F^{x,n,c}_t)_{t \in \R_+}$ and
$\dot{\F}^{x,n,c} \defeq (\dot{F}^{x,n,c}_t)_{t\in\R}$ to be the filtrations to which $N^{x,n,c}$ and $\dot{N}^{x,n,c}$ are adapted to respectively, with
$\F^{x,n,c}_t \defeq \dot{\F}^{x,n,c}_t \defeq \F^n_{\zeta^{x,n}_t}$. $N^{x,n,c}$ and $\dot{N}^{x,n,c}$ have, respectively, intensities $\lambda^{x,n,c}_{t, \theta^\ast}$ and
$\dot{\lambda}^{x,n,c}_{t, \theta^\ast}$, where for any $\theta$, we have
\begin{align*}
  \lambda^{x,n,c}_{t, \theta} \defeq {} & \nu(x; \theta) + \int_0^{t-} g(t-s; \theta) \rmd N^{x,n,c}_s, \quad t \in \R_+;\\
  \dot{\lambda}^{x,n,c}_{t, \theta} \defeq {} & \nu(x; \theta) + \int_{-\infty}^{t-} g(t-s; \theta) \rmd \dot{N}^{x,n,c}_s, \quad t \in \R.
\end{align*}
Let
$Y^{x,n}_{t, \theta, \theta^\ast} \defeq (\lambda^{x,n}_{t, \theta}, \lambda^{x,n}_{t, \theta^\ast}, \partial_\theta
\lambda^{x,n}_{t, \theta}, \partial^{\otimes 2}_\theta \lambda^{x,n}_{t, \theta})$,
$Y^{x,n,c}_{t, \theta, \theta^\ast} \defeq (\lambda^{x,n,c}_{t, \theta}, \lambda^{x,n,c}_{t, \theta^\ast},
\partial_\theta \lambda^{x,n,c}_{t, \theta}, \partial^{\otimes 2}_\theta \lambda^{x,n,c}_{t, \theta})$ and
$\dot{Y}^{x,n,c}_{t, \theta, \theta^\ast} \defeq (\dot{\lambda}^{x,n,c}_{t, \theta}, \dot{\lambda}^{x,n,c}_{t,
  \theta^\ast}, \partial_\theta \dot{\lambda}^{x,n,c}_{t, \theta}, \partial^{\otimes 2}_\theta \dot{\lambda}^{x,n,c}_{t,
  \theta})$ all take values in $E_1 = \R_+ \times \R_+ \times \R^d \times \R^{d \times d}$. We can now present the proof
for Theorem~\ref{thm:asympergodicity}.

\begin{proof}[Proof of Theorem~\ref{thm:asympergodicity}]
  Note that
  \begin{equation*}
    \frac{1}{n} \int_0^n \psi(Y^n_{t, \theta, \theta^\ast}) \rmd t = \frac{1}{a_n} \sum_{i=1}^{a_n} \frac{1}{b_n} \int_0^{b_n} \psi(Y^{i,n}_{t, \theta, \theta^\ast}) \rmd t.
  \end{equation*}	
  Hence, if there exists a mapping $\dot{\pi}:[0,1] \times \Theta \times \Theta \times D_p(E_1, \R) \to \R$ such that
  \begin{equation}\label{eq:asympergodicityeq2}
    \frac{1}{a_n} \sum_{i=1}^{a_n} \left\vert \frac{1}{b_n} \int_0^{b_n} \psi (Y^{i,n}_{t, \theta, \theta^\ast}) \rmd t - \dot{\pi} (\frac{\tau^n_{i-1}}{n}, \theta, \theta^\ast; \psi) \right\vert \overset{\PP}{\rightarrow} 0, \quad  n \rightarrow \infty,
  \end{equation}
  and $\dot{\pi}(x, \theta, \theta^\ast; \psi)$ is integrable with respect to $x$, then \eqref{eq:asympergodicityeq1}
  holds with $\pi(\theta, \theta^\ast; \psi) = \int_0^1 \dot{\pi}(x, \theta, \theta^\ast; \psi)\rmd
  x$. \eqref{eq:asympergodicityeq2} is shown via the aforementioned approximation procedure, the likes of which have
  also been seen in the works by \cite{clinet2018statistical} and \cite{kwan2022alternative}. That is,
  \eqref{eq:asympergodicityeq2} holds if
  \begin{align}
    & \sup_{x \in [0,1]} \left \vert \frac{1}{b_n}\int_0^{b_n} \psi (Y^{x,n}_{t, \theta, \theta^\ast}) - \psi(Y^{x,n,c}_{t, \theta, \theta^\ast}) \rmd t \right\vert \overset{\PP}{\rightarrow} 0,  \label{eq:asympergodicityeq3} \\
    &  \sup_{x \in [0,1]} \left \vert \frac{1}{b_n}\int_0^{b_n} \psi (Y^{x,n,c}_{t, \theta, \theta^\ast}) - \psi(\dot{Y}^{x,n,c}_{t, \theta, \theta^\ast}) \rmd t \right\vert \overset{\PP}{\rightarrow} 0, \label{eq:asympergodicityeq4}\\
    \shortintertext{and}
    & \sup_{x \in [0,1]} \left \vert \frac{1}{b_n}\int_0^{b_n} \psi (\dot{Y}^{x,n,c}_{t, \theta, \theta^\ast}) \rmd t- \dot{\pi}(x, \theta, \theta^\ast; \psi) \right\vert \overset{\PP}{\rightarrow} 0. \label{eq:asympergodicityeq5}
\end{align}
The following lemmas ensure \eqref{eq:asympergodicityeq3} – \eqref{eq:asympergodicityeq5} hold, the proofs of which can
be found in Appendix~\ref{sec:ErgodicTheory-prep_Proofs}.
\begin{lemma}\label{lm:ergodicity-lm1} Under \ref{cond:C1}, \ref{cond:C2} and \hyperref[cond:C3]{[C3]($p$)},
\eqref{eq:asympergodicityeq3} holds.
\end{lemma}
\begin{lemma}\label{lm:ergodicity-lm2} Under \ref{cond:C1}, \ref{cond:C2} and \hyperref[cond:C3]{[C3]($p$)},
\eqref{eq:asympergodicityeq4} holds.
\end{lemma}
\begin{lemma}\label{lm:ergodicity-lm3} Under \ref{cond:C1}, \ref{cond:C2} and \hyperref[cond:C3]{[C3]($p+\epsilon$)},
$\dot{Y}^{x,n,c}_{t, \theta, \theta^\ast}$ is $C_p(E_1, \R)$-ergodic and \eqref{eq:asympergodicityeq5}
holds. Furthermore, we have $\dot{\pi}(x, \theta, \theta^\ast; \psi) = \E[\psi(\dot{Y}^{x,n,c}_{t,\theta,
\theta^\ast})]$.
\end{lemma} To complete the proof, it remains to show $\dot{\pi}(x, \theta, \theta^\ast; \psi)$ is integrable in
$x$. This is achieved by the following lemma, the proof of which is relegated to
Appendix~\ref{sec:ErgodicTheory-prep_Proofs}.
 \begin{lemma}\label{lm:ergodicity-lm4} Under \ref{cond:C1}, \ref{cond:C2} and \hyperref[cond:C3]{[C3]($p+\epsilon$)},
for any $\psi \in D_p(E_1, \R)$, $\dot{\pi}(x, \theta, \theta^\ast; \psi)$ is equicontinuous in $x$ when regarded as
family of functions of $\theta, \theta^\ast$ and $\psi$.
 \end{lemma}
\end{proof}

\begin{remark}
  Note that \eqref{eq:asympergodicityeq3} – \eqref{eq:asympergodicityeq5} outlines the aforementioned approximation
  procedure. The idea is to identify two pseudo processes, $Y^{x,n,c}_{t, \theta, \theta^\ast}$ and
  $\dot{Y}^{x,n,c}_{t, \theta, \theta^\ast}$, such that the integral
  $\frac{1}{n} \int_0^n \psi(Y^n_{t, \theta, \theta^\ast})\rmd t$ on each block can be asymptotically approximated by
  $\frac{1}{b_n}\int_0^{b_n} \psi(Y^{x,n,c}_{t, \theta, \theta^\ast}) \rmd t$ (represented by
  \eqref{eq:asympergodicityeq2}) and $\frac{1}{b_n}\int_0^{b_n} \psi(Y^{x,n,c}_{t, \theta, \theta^\ast}) \rmd t$ by
  $\frac{1}{b_n}\int_0^{b_n} \psi(\dot{Y}^{x,n,c}_{t, \theta, \theta^\ast}) \rmd t$ (represented by
  \eqref{eq:asympergodicityeq3}) uniformly in $x$. Furthermore, suppose $\dot{Y}^{x,n,c}_{t, \theta, \theta^\ast}$ is
  $D_p(E_1, \R)$-ergodic (represented by \eqref{eq:asympergodicityeq4}), then by imposing continuity type conditions on
  $\nu$, the ergodic limit on each block can be aggregated such that
  $\pi(\theta, \theta^\ast; \psi) = \int_0^1 \dot{\pi}(x, \theta, \theta^\ast; \psi)\rmd x$ exists as the asymptotic
  ergodic limit of $\frac{1}{n}\int_0^n \psi(Y^n_{t, \theta, \theta^\ast}) \rmd t$ (represented by
  \eqref{eq:asympergodicityeq5}).
\end{remark}

As we shall see in the next section, if we strengthen the convergence in \eqref{eq:asympergodicityeq1} to be uniform in
$\theta$ and along with identifiability type conditions, we can establish consistency and asymptotic normality of the
MLE. Generalising the convergence in \eqref{eq:asympergodicityeq1} to be uniform in $\theta$ is achieved by imposing the
following conditions regarding the continuity of the intensity processes in $\theta$.
\begin{enumerate}[label= {[C4]}($p$) ]
\item \label{cond:C4}
  \begin{enumerate}
  \item [{(i)}] For any $j \in \{0, 1, 2\}$, the family of functions
    $\{\partial_\theta^{\otimes j} \nu(x; \theta)\}_{x \in [0,1]}$ is equicontinuous in $\theta$.
  \item [{(ii)}]For any $j \in \{0, 1, 2\}$, the family of functions
    $\{ \partial_\theta^{\otimes j}g (t; \theta)\}_{t \in \R_+}$ is equicontinuous in $\theta$. Furthermore, for any
    $\theta \in \Theta$, $\vert \partial_\theta^{\otimes 3} g(t; \theta) \vert^p$ is integrable with respect to $t$.
	\end{enumerate}
\end{enumerate}

The following proposition generalises the convergence in \eqref{eq:asympergodicityeq1} to be uniform in $\theta$, the
proof of which is relegated to Appendix~\ref{sec:ErgodicTheory-prep_Proofs}.
\begin{proposition}\label{prop:uniform-convergence}
  In addition to the conditions stated in Theorem~\ref{thm:asympergodicity}, under \hyperref[cond:43]{[C4]($p$)}, the
  convergence in \eqref{eq:asympergodicityeq1} is uniform in $\theta$.
\end{proposition}

\section{Parametric Inference}\label{sec:parametricinference}
In this section, we make use of the results established in Section~\ref{sec:main_results} to show the MLE is consistent
and asymptotically normal. The normalised log-likelihood, score function and negative Hessian matrix of $\theta$
relative to the sample path of $N^n$ $[0, n]$ are given respectively in \eqref{LogLikelihood}, \eqref{ScoreFunction} and
\eqref{InformationMatrix}. To establish consistency and asymptotic normality of the MLE $\hat{\theta}^n$, $\theta^\ast$
needs to be identifiable. For this reason, we require the following conditions.
\begin{enumerate}[label= {[C5]} ]
\item \label{cond:C5}
  \begin{enumerate}
  \item [{(i)}] The parameters for $\nu(\cdot)$ and $g(\cdot)$ are separable. That is,
    $\theta = (\theta_\nu\tr, \theta_g\tr)\tr \in \Theta_\nu \times \Theta_g = \Theta \subset \R^{d_\nu + d_g}$ where
    $\theta_\nu = (\theta_{\nu, 1}, \ldots, \theta_{\nu, d_\nu})\tr$ and
    $\theta_g =(\eta, \theta_{g,2}, \ldots, \theta_{g, d_g})\tr$, with $\nu(\cdot; \theta_\nu)$ depending on
    $\theta_\nu$ and $g(\cdot; \theta_g)$ depending only on $\theta_g$. In particular, we have
    $g(t; \theta_g) = \eta \tilde{g}(t; \tilde{\theta}_g)$ where $\tilde{g}$ is a proper density function of a
    non-negative random variable that depends on $\tilde{\theta}_g \in \R^{d_g-1}$ only. Furthermore, we assume that
    $\Theta_\nu$ and $\Theta_g$ are both compact and contain a nonempty open ball in $\R^{d_v}$ and $\R^{d_g}$
    respectively, which in turn contains $\theta_\nu^\ast$ and $\theta_g^\ast$ respectively, with
    $(\theta_\nu^{\ast \top}, \theta_g^{\ast\top})\tr = \theta^\ast$ being the true parameter.
  \item [{(ii)}] $\nu(x; \theta_\nu) = \nu(x; \theta_\nu')$ for all $x \in [0,1]$ only when $\theta_\nu =
    \theta_\nu'$. Furthermore, for all $\theta_\nu \in \Theta_\nu$,
    $\sum_{i=1}^{d_\nu} z_{\nu, i}\partial_{\theta_{\nu, i}}\nu(x; \theta_\nu) = 0$ for all $x \in [0,1]$ only when
    $z_{\nu, i} = 0$ for all $i \in \{1, \ldots, d_\nu\}$.
  \item [{(iii)}] $\tilde{g}(t; \tilde{\theta}_g) = \tilde{g}(t; \tilde{\theta}_g')$ for all $t \in \R_+$ only when
    $\tilde{\theta}_g = \tilde{\theta}_g'$. Furthermore, for all $\theta_g \in \Theta_g$,
    $\sum_{i=1}^{d_g} z_{g, i} \partial_{\theta_{g, i}}g(t; \theta_g) = 0$ for all $t \in \R_+$ only when $z_{g, i}=0$
    for all $i \in \{1, \ldots, d_g\}$.
  \end{enumerate}
\end{enumerate}
\subsection{Consistency}\label{subsec:consistency}
We can now present the theorem on consistency of the MLE.
\begin{theorem}\label{thm:consistency}
  Under \ref{cond:C1}, \ref{cond:C2}, \hyperref[cond:C3]{[C3](2)}, \hyperref[cond:C4]{[C4](2)} and \ref{cond:C5}, $\hat{\theta}^n$ exists as a maximiser to the normalised log-likelihood function
  $\mathcal{L}^n(\theta)$ and $\hat{\theta}^n \rightarrow \theta^\ast$ in probability.
\end{theorem}
\begin{proof}
  To establish consistency of the MLE, we make use of the ergodic property established in
  Theorem~\ref{thm:asympergodicity} to, first, show that there exists some mapping $\pi: \Theta \times \Theta \to \R$
  such that
  \begin{equation}\label{eq:consistency-proof-eq1}
    \sup_{\theta \in \Theta} \vert \mathcal{L}^n(\theta) - \mathcal{L}^n(\theta^\ast) - \pi(\theta, \theta^\ast) \vert \overset{\mathbb{P}}{\rightarrow} 0.
  \end{equation}
  Next, establish $\theta = \theta^\ast$ exists as a well-separated point of maximum to $\pi(\theta, \theta^\ast)$, in
  the sense that for any $\epsilon >0$, there exists an open ball $B(\theta^\ast, \epsilon)$ centred at $\theta^\ast$
  with radius $\epsilon$ such that
  $\sup_{\theta \in \Theta \backslash B(\theta^\ast, \epsilon)} \pi(\theta, \theta^\ast) < \pi (\theta^\ast,
  \theta^\ast)$. Then, by the M-estimator master theorem (\cite{van2000asymptotic}, Theorem 5.7), we have
  $\hat{\theta}^n \rightarrow \theta^\ast$ in probability and we can conclude that the MLE $\hat{\theta}^n$ is a
  consistent estimator.

  To show \eqref{eq:consistency-proof-eq1}, note that
  $\vert \mathcal{L}^n(\theta) - \mathcal{L}^n(\theta^\ast) - \pi(\theta, \theta^\ast)\vert$ is bounded by
  \begin{equation}
    \label{eq:consistency-proof-eq2}
    \vert M^n(\theta, \theta^\ast) \vert + \vert L^n(\theta, \theta^\ast) - L^n(\theta^\ast, \theta^\ast) - \pi(\theta, \theta^\ast) \vert
  \end{equation}
  where
  $M^n(\theta, \theta^\ast) = \frac{1}{n} \int_0^n \log (\lambda^n_{t, \theta} / \lambda^n_{t, \theta^\ast}) \rmd
  \bar{N}^n_t$ with $\bar{N}^n_t = N^n_t - \int_0^t \lambda^n_{s, \theta^\ast} \rmd t$. By
  Theorem~\ref{thm:asympergodicity} and Proposition~\ref{prop:uniform-convergence}, under the assumed conditions, the
  second term of \ref{eq:consistency-proof-eq2} converges to 0 in probability uniformly in $\theta$. As for
  $M^n(\theta, \theta^\ast)$, note that
  $\frac{1}{n}\int_0^t \log(\lambda^n_{t, \theta}/ \lambda^n_{t, \theta^\ast}) \rmd \bar{N}^n_s$ is a local square
  integrable martingale. Hence, by Lenglart's inequality (\cite{andersen2012statistical}, II.5.2.1, page 86), for any
  $\tilde{\epsilon}$ and $\delta > 0$, we have
  \begin{equation}
    \label{eq:consistency-proof-eq3}
    \mathbb{P} \left( \sup_{t \in [0,n]} \left\vert \frac{1}{n} \int_0^t \log \left( \frac{\lambda^n_{s, \theta}}{\lambda^n_{s, \theta^\ast}} \right) \rmd
        \bar{N}^n_s \right\vert > \tilde{\epsilon} \right) \le \frac{\delta}{\tilde{\epsilon}^2} + \mathbb{P} \left( \frac{1}{n^2} \int_0^n \log \left(
        \frac{\lambda^n_{s, \theta}}{\lambda^n_{s, \theta^\ast}}\right)^2 \lambda^n_{s, \theta^\ast} \rmd s > \delta \right).
  \end{equation}
  By Lemma~\ref{lm:prep-proofs-lm1}, the second term of the above bound converges to 0. Furthermore, since
  $\tilde{\epsilon}$ and $\delta$ are arbitrary, we deduce that $M^n(\theta, \theta^\ast$ converges to 0 in
  probability. To see that this convergence is uniform in $\theta$, by the assumed continuity conditions under
  \hyperref[cond:C4]{[C4](2)}, for any $\tilde{\epsilon} >0$, there exists $\delta >0$ such that for all
  $\theta, \theta' \in \Theta$, $\vert \theta - \theta' \vert <\delta$ implies
  $ \vert M^n(\theta, \theta^\ast) - M^n(\theta', \theta^\ast)\vert \le \tilde{\epsilon}$ with probability tending to
  1. As $\Theta$ is compact, thus by following the proof of Proposition~\ref{prop:uniform-convergence},
  $M^n(\theta, \theta^\ast) \rightarrow 0$ in probability uniformly in $\theta$.

  It remains to show that $\theta^\ast$ is a well-separated point of maximum of $\pi(\theta, \theta^\ast)$. By
  Theorem~\ref{thm:asympergodicity} and Lemma~\ref{lm:ergodicity-lm3},
  $\pi(\theta, \theta^\ast) = \int_0^1 \dot{\pi}(x, \theta, \theta^\ast) \rmd x$. Noting that
  $\dot{\pi}(x, \theta, \theta^\ast)=0$ for all $x \in [0,1]$. Thus, if we can show that for any $\tilde{\epsilon} > 0$,
  \begin{equation}
    \label{eq:consistency-proof-eq4}
    \sup_{\theta \in \Theta : \vert \theta - \theta^\ast \vert > \tilde{\epsilon}} \dot{\pi} (x, \theta, \theta^\ast) <0
  \end{equation}
  for all $x \in [0,1]$, then $\theta^\ast$ exists as a well-separated point of maximum of $\pi(\theta,\theta^\ast)$. Since
  $\pi(x, \theta, \theta^\ast) = \E[\log(\dot{\lambda}^x_{t, \theta}/\dot{\lambda}^x_{t, \theta^\ast}) \dot{\lambda}^x_{t, \theta} - \dot{\lambda}^x_{t, \theta} + \dot{\lambda}^x_{t, \theta^\ast}]$,
  we need only to show for all $x \in [0,1]$ and $t \in \R$, $\dot{\lambda}^x_{t, \theta} = \dot{\lambda}^x_{t, \theta'}$ almost surely if and only if $\theta = \theta'$. Note that when
  $\theta = \theta'$, it trivally holds that $\dot{\lambda}^x_{t, \theta} = \dot{\lambda}^x_{t, \theta'}$ almost surely. Hence, we need only to show
  $\dot{\lambda}^x_{t, \theta} = \dot{\lambda}^x_{t, \theta'}$ almost surely implies $\theta = \theta'$. A quick calculation shows that
  $\E[\dot{\lambda}^x_{t, \theta^\ast}] = \nu(x; \theta^\ast_\nu)/(1-\eta^\ast)$. Furthermore, assuming $\dot{\lambda}^x_{t, \theta} = \dot{\lambda}^x_{t, \theta'}$ almost surely, we have
  $\E[\dot{\lambda}^x_Pt, \theta] = \E[\dot{\lambda}^x_{t, \theta'}]$. By Fubini's theorem, we have
\begin{align*}
  \nu(x; \theta_\nu) - \nu(x; \theta'_\nu) = {} & \int_{-\infty}^{t-} \big( g(t-s; \theta_g') - g(t-s; \theta_g) \big) \E[\dot{\lambda}^x_{s, \theta^\ast}]\rmd s \\
  = {} & \frac{\nu(x; \theta_\nu^\ast)}{1-\eta^\ast} (\eta' - \eta).
\end{align*}
Since the parameters are separable, $\nu(x; \theta_\nu) - \nu(x - \theta_\nu')$ is dependent of $\eta$ and $\eta'$. Furthermore, under \ref{cond:C5}, there exists $\theta''_\nu \in \Theta_\nu$ such
that $\nu(x; \theta_\nu) - \nu(x; \theta'_\nu) \ne \nu(x; \theta_\nu) - \nu(x;\theta''_\nu)$, thus we deduce that $\eta = \eta'$. This also implies $\nu(x; \theta_\nu) = \nu(x; \theta'_\nu)$ and
therefore, as a consequence of \ref{cond:C5}-(ii), $\theta_\nu = \theta'_\nu$. We have shown that $\theta_\nu = \theta'_\nu$ and $\eta = \eta'$. Hence if
$\dot{\lambda}^x_{t, \theta} = \dot{\lambda}^x_{t, \theta'}$ almost surely, we also have, almost surely,
\begin{equation}
  \label{eq:consistency-proof-eq5}
  0 = \int_{-\infty}^{t-} \eta \big( \tilde{g}(t-s; \tilde{\theta}'_g) - \tilde{g}(t-s;\tilde{\theta}_g)\big) \rmd \dot{N}^x_s.
\end{equation}
Let $f(t) = \tilde{g}(t-s; \tilde{\theta}'_g) - \tilde{g}(t-s;\tilde{\theta}_g) = 0$. Note that if $f(t) = 0$ for all $t \in \R_+$, then in view of \hyperref[cond:C5]{[C5]-(iii)}, we have
$\tilde{\theta}_g = \tilde{\theta}'_g$. This can be shown via contradiction. Suppose there exists $t \in \R_+$ such that $f(t) \ne 0$. Since $f$ is continuous, there exist $\underline{t}$,
$\bar{t} \in \R_+$ such that $(\underline{t}, \bar{t}) \ne \emptyset$ and $f(t) \ne 0$ for all $t \in (\underline{t}, \bar{t})$. By noting that the probability of observing at least one event over this
nonempty open interval is positive, we have a contradiction. This shows $f(t) = 0$ for all $t \in \R_+$.
\end{proof}
\begin{remark}
  \label{remark:consistency-conditions}
  If $\lambda^n_{t, \theta}$, $\lambda^{x,n,c}_{t, \theta}$ and $\dot{\lambda}^{x,n,c}_{t, \theta}$ are all bounded in $\mathbb{L}_{1+\epsilon}$ for some $\epsilon > 0$ uniformly in $t$ and
  $\theta$, then we no longer require integrability of $g$ as assumed under \hyperref[cond:C3]{[C3](2)} for $L^n(\theta, \theta^\ast) - L^n(\theta^\ast, \theta^\ast) - \pi(x, \theta, \theta^\ast)
  \rightarrow 0$ in probability. This stems from the fact that Lemma~\ref{lm:prep-proofs-lm1} also holds for non-integer valued $p$ as a consequence of the BDG inequality. For more details see Remark~\ref{remark:prep-proofs-remark1}.
\end{remark}

\subsection{Asymptotic Normality}\label{subsec:asym_norm}
On the asymptotic distribution of the MLE, we have the following result.
\begin{theorem}\label{thm:asym-norm}
	Under \ref{cond:C1}, \ref{cond:C2}, \hyperref[cond:C3]{[C3](5)}, \hyperref[cond:C4]{[C4](3)} and \ref{cond:C5}, there exists an invertible matrix $\Gamma(\theta^\ast) \in \R^{d \times d}$ such that
	\begin{equation}\label{eq:asynormeq1}
		\sqrt{n} (\hat{\theta}^n - \theta^\ast) \overset{\mathcal{D}}{\rightarrow} \mathcal{N}(0, \Gamma(\theta^\ast)^{-1})
	\end{equation}
      \end{theorem}
\begin{proof}
  Let $\tilde{\theta}^n$ be a point that lies on the line segment joining $\hat{\theta}^n$ and $\theta^\ast$, an
  application of the mean value theorem to $S^n$ gives
\begin{equation*}
  \sqrt{n} S^n(\hat{\theta}^n) = \sqrt{n} S^n(\theta^\ast) - \mathcal{I}^n(\tilde{\theta}^n)
  \sqrt{n}(\hat{\theta}^n - \theta^\ast).
\end{equation*}
Since $\sqrt{n}S(\hat{\theta}^n) = 0$, we have
\begin{equation*}
  \sqrt{n} (\hat{\theta}^n - \theta^\ast ) = \mathcal{I}^n(\tilde{\theta}^n)^{-1}
  \sqrt{n}S^n(\theta^\ast).
\end{equation*}
In view of Slutsky's theorem, if we can show that
\begin{equation}\label{eq:asym_norm_pf_eq1}
  \sqrt{n} S^n(\theta^\ast) \converginD \mathcal{N}\big(0,\Gamma(\theta^\ast) \big)
\end{equation}
and
\begin{equation}\label{eq:asym_norm_pf_eq2}
  \mathcal{I}^n(\tilde{\theta}^n) \converginP \Gamma(\theta^\ast)
\end{equation}
for some matrix valued function $\Gamma(\cdot)$ that is invertible when evaluated at $\theta^\ast$, we have the desired
result.
	
Let us show \eqref{eq:asym_norm_pf_eq1} and \eqref{eq:asym_norm_pf_eq2} sequentially. There exists a martingale
representation of the score function where we can write $\sqrt{n} S^n(\theta^\ast) = \int_0^n U^n_s \rmd \bar{N}^n_s$
with $U^n_s = \partial_{\theta^\ast} \lambda^n_{s, \theta^\ast} / (\sqrt{n} \lambda^n_{s, \theta^\ast})$.  The process
$U^n_t = (U^{1,n}_t, \ldots, U^{d,n}_t)\tr \in \R^d$ is locally bounded and predictable, and the compensated process
$\bar{N}^n$ is a square-integrable local martingale. Hence, under the assumed conditions, each component of
$M^n_t = (M^{1,n}_t, \ldots, M^{d,n}_t)\tr = \int_0^{nt} U^n_s \rmd \bar{N}^n_s \in \R^d$ is a square-integrable local
martingale with $M^n_1 = \sqrt{n}S^n(\theta^\ast)$. Denote
$M^{n, \tilde{\epsilon}}_t = (M^{1,n,\tilde{\epsilon}}_t, \ldots, M^{d,n,\tilde{\epsilon}}_t)\tr \in \R^d$ where
$M^{i,n,\tilde{\epsilon}}_t = \int_0^{nt} U^{i,n}_s \1 \{ \vert U^{i,n}_s \vert > \tilde{\epsilon}\}\rmd \bar{N}^n_s$
for $i \in \{1, \ldots, d\}$. If there exists a mapping $\gamma:[0,1] \times \Theta \mapsto \R^{d \times d}$ such that

\begin{equation}\label{eq:asym_norm_pf_eq3}
  \langle M^n, M^n \rangle_t \converginP \gamma(t, \theta^\ast)
\end{equation}
and
\begin{equation}\label{eq:asym_norm_pf_eq4}
  \langle M^{n,\tilde{\epsilon}}, M^{n, \tilde{\epsilon}} \rangle_t \converginP 0 \quad \forall \ \tilde{\epsilon} >0,
\end{equation} 
then, by the martingale central limit theorem (see \cite{fleming2011counting}, Theorem 5.1.1),
\eqref{eq:asym_norm_pf_eq1} holds with $\Gamma(\theta^\ast) = \gamma(1, \theta^\ast)$. To see
\eqref{eq:asym_norm_pf_eq3} holds, note that
$\langle \bar{N}^n, \bar{N}^n \rangle_t = \int_0^t \lambda^n_{s, \theta^\ast} \rmd s$. Hence, the predictable variation
of $M^n_t$ is given by
\begin{equation*}
  \langle M^n, M^n \rangle_t = \frac{1}{n} \int_0^{nt} \frac{\big(\partial_{\theta^\ast} \lambda^n_{s, \theta^\ast} \big)^{\otimes 2}}{\lambda^n_{s, \theta^\ast}} \rmd s.
\end{equation*}
Since $t \in [0,1]$, by adopting the approximation procedure outlined in Section~\ref{sec:main_results}, we can split
the integral into $a_n$ number of equal width blocks with width $b_n = n/a_n$ and write
\begin{align*}
  \langle M^n, M^n \rangle_t = \frac{1}{a_n} \sum_{i=1}^{\lfloor t a_n \rfloor} \frac{1}{b_n} \int_0^{b_n} \frac{\big( \partial_{\theta^\ast} \lambda^{i,n}_{s, \theta^\ast} \big) ^{\otimes 2}}{\lambda^{i,n}_{s, \theta^\ast} } \rmd s + \frac{1}{n} \int_0^{nt - \lfloor t a_n \rfloor b_n} \frac{\big( \partial_{\theta^\ast} \lambda^{\lfloor t a_n \rfloor + 1, n}_{s, \theta^\ast}\big)^{\otimes 2} }{\lambda^{\lfloor t a_n \rfloor + 1, n}_{s,\theta^\ast}} \rmd s.
\end{align*}
The second term on the right converges to 0 almost surely since $nt - \lfloor t a_n \rfloor b_n = o(n)$. From the proof
of Theorem~\ref{thm:asympergodicity}, we deduce that under the assumed conditions, there exists
$\dot{\gamma}:[0,1] \times \Theta \mapsto \R^{d \times d}$ such that
\begin{equation*}
  \sup_{x \in [0,1]} \Big\vert \frac{1}{b_n} \int_0^{b_n} \frac{\big( \partial_{\theta^\ast} \lambda^{x,n}_{t, \theta^\ast} \big)^{\otimes 2}}{\lambda^{x,n}_{t, \theta^\ast}} \rmd t - \dot{\gamma}(x, \theta^\ast) \Big\vert \converginP 0,
\end{equation*}
from which we can further deduce that
\begin{equation*}
  \Big\vert \frac{1}{a_n} \sum_{i=1}^{\lfloor t a_n \rfloor}  \Big( \frac{1}{b_n} \int_0^{b_n} \frac{\big( \partial_{\theta^\ast} \lambda^{i,n}_{s, \theta^\ast} \big)^{\otimes 2} }{\lambda^{i,n}_{s, \theta^\ast} } \rmd s- \dot{\gamma}\big( \frac{\tau^n_{i-1}}{n}, \theta^\ast \big) \Big) \Big\vert \converginP 0
\end{equation*}
for all $t \in[0,1]$. Following the proof of Lemma~\ref{lm:ergodicity-lm4}, it can be shown that
$\{\dot{\gamma}(x, \theta) \}_{\theta \in \Theta}$ is equicontinuous, and, hence, is integrable with respect to
$x$. This shows \eqref{eq:asym_norm_pf_eq3} holds with
$\gamma(t, \theta^\ast) = \int_0^t \dot{\gamma}(x, \theta^\ast) \rmd x$ and
$\Gamma(\theta^\ast) = \gamma(1, \theta^\ast)$.
	
Next, we turn to the big jumps of $M^n_t$. Note that $M^{n, \tilde{\epsilon}}_t$ contains all the jumps of $M^n$ of size
$\tilde{\epsilon}$ or larger. Hence, it suffices that
$\sup_{s \in [0, t]} \vert M^{n, \tilde{\epsilon}}_s \vert \rightarrow 0$ in probability dimension-wise for all
$t \in [0,1]$ in order to deduce \eqref{eq:asym_norm_pf_eq4}. By Lenglart's inequality again, for any $\eta >0$, there
exists $\delta >0$ such that
\begin{equation*}
  \PP \big( \sup_{s \in [0,t]} \big\vert M^{i, n, \tilde{\epsilon}}_t \big\vert > \eta \big) \le \frac{\delta}{\eta^2} + \PP \Big( \int_0^{nt} \big( U^{i,n}_s \1 \{ \vert U^{i,n}_s \vert > \tilde{\epsilon} \} \big)^2 \lambda^n_{s, \theta^\ast} \rmd s > \delta \Big).
\end{equation*}
Under the assumed conditions and by Lemma~\ref{lm:prep-proofs-lm1},
$(U^{i,n}_s)^2 \lambda^n_{s, \theta^\ast} = O_P(n^{-1})$ for all $s \in [0, n]$. Hence, the second term of the above
bound tends to 0 with $n$ for all $t \in [0,1]$. Since $\delta$ and $\eta$ are arbitrary, this shows
$\sup_{s \in [0,t]} \vert M^{n, \tilde{\epsilon}}_s \vert \rightarrow 0$ in probability dimension-wise for all
$t \in [0,1]$. This proves \eqref{eq:asym_norm_pf_eq4} and, consequently, \eqref{eq:asym_norm_pf_eq1}.
	
To show \eqref{eq:asym_norm_pf_eq2}, note that we can write
$\mathcal{I}^n(\theta) = J^n(\theta) + I^n(\theta, \theta^\ast)$ where
\begin{equation*}
  J^n(\theta) = \frac{-1}{n} \int_0^n \frac{\partial_\theta^{\otimes 2} \lambda^n_{t, \theta}}{\lambda^n_{t, \theta}} - \bigg(\frac{\partial_\theta \lambda^n_{t, \theta} }{\lambda^n_{t, \theta}} \bigg)^{\otimes 2} \rmd \bar{N}^n_t
\end{equation*}
and
\begin{equation*}
  I^n(\theta, \theta^\ast) = \frac{1}{n} \int_0^n - \bigg( \frac{\partial_\theta^{\otimes 2} \lambda^n_{t, \theta}}{\lambda^n_{t, \theta}} - \bigg( \frac{\partial_\theta \lambda^n_{t, \theta}}{\lambda^n_{t, \theta}}\bigg)^{\otimes 2} \bigg) \lambda^n_{t, \theta^\ast} + \partial_\theta^{\otimes 2} \lambda^n_{t, \theta} \rmd t.
\end{equation*}
By Lenglart's inequality again, it can be shown that $\vert J^n(\theta)\vert \rightarrow 0$ in probability under
\hyperref[cond:C3]{[C3]($5$)}. By Theorem~\ref{thm:asympergodicity} and Proposition~\ref{prop:uniform-convergence},
there exists $\alpha: \Theta \times \Theta \mapsto \R^{d \times d}$ such that
\begin{equation*}
  \sup_{\theta \in\Theta} \big\vert I^n(\theta, \theta^\ast) - \alpha (\theta, \theta^\ast) \big\vert \converginP 0.
\end{equation*}
Since $\hat{\theta}^n \rightarrow \theta^\ast$ in probability and $\tilde{\theta}^n$ is sandwiched between
$\hat{\theta}^n$ and $\theta^\ast$, we have $\tilde{\theta}^n \rightarrow \theta^\ast$ in probability. Following the
arguments from the proof of Lemma~\ref{lm:ergodicity-lm4}, it can be shown that $\alpha(\theta, \theta^\ast)$ is
continuous in $\theta$. Hence by the continuous mapping theorem,
$\alpha(\hat{\theta}^n, \theta^\ast) \rightarrow \alpha(\theta^\ast, \theta^\ast)$ in probability.  Furthermore, noting
that $\mathcal{I}^n(\theta^\ast) = \langle M^n, M^n \rangle_1 \rightarrow \Gamma(\theta^\ast)$ in probability, it
follows that $\alpha(\theta^\ast, \theta^\ast) = \Gamma(\theta^\ast)$.
	
It remains to show $\Gamma(\theta^\ast)$ is positive definite. Recall that
$\Gamma(\theta^\ast) = \int_0^1 \dot{\gamma}(x, \theta^\ast) \rmd x$, by Lemma~\ref{lm:ergodicity-lm3}, the mapping
$\dot{\gamma}(x, \theta^\ast)$ takes the form
\begin{equation*}
  \dot{\gamma}(x, \theta^\ast) = \E \bigg[ \frac{\big( \partial_{\theta^\ast} \dot{\lambda}^x_{t, \theta^\ast} \big)^{\otimes 2} }{\dot{\lambda}^x_{t, \theta^\ast} }\bigg].
\end{equation*}
For some fixed $z = (z_1, \ldots, z_d)\tr \in \R^d$, we shall show that $z\tr \dot{\gamma}(x, \theta^\ast) z = 0$ only
when $z = 0$. Rewriting $z\tr \dot{\gamma}(x, \theta^\ast) z = 0$ as
\begin{equation*}
  \E \bigg[ \frac{1}{\dot{\lambda}^x_{t, \theta^\ast}} \Big( \sum_{i=1}^d z_i  \partial_{\theta_i}\dot{\lambda}^x_{t, \theta^\ast}  \Big)^2 \bigg] =0,
\end{equation*} 
and noting that $0 < \underline{\nu} \le \dot{\lambda}^x_{t, \theta^\ast}$ almost surely, it suffices to show
\begin{equation*}
  \sum_{i=1}^d z_i \partial_{\theta_i}\dot{\lambda}^x_{t, \theta^\ast}  =0
\end{equation*}
almost surely only when $z_i= 0$ for all $i \in \{1, \ldots, d\}$. By inspection, when $z_i = 0$ for all $i$, we have
$$\sum_{i=1}^d z_i \partial_{\theta_i} \dot{\lambda}^x_{t, \theta} = 0$$
almost surely. To see that the converse holds, since
$\E[\sum_{i=1}^d z_i \partial_{\theta_i} \dot{\lambda}^x_{t, \theta} ] =0$ and
$\E [\dot{\lambda}^x_{t, \theta^\ast}] = \nu(x; \theta^\ast) /(1- \eta^0)$, thus, by Leibniz's rule for integrals, we
can write
        
\begin{align*}
  \sum_{i=1}^{d_\nu} z_{\nu, i} \partial_{\theta_{\nu,i}} \nu(x; \theta_\nu)	= {} & - \int_{-\infty}^{t-} \sum_{i=1}^{d_g} z_{g, i} \partial_{\theta_{g, i}} g(t-s; \theta_g) \E \big[ \dot{\lambda}^x_{s, \theta^\ast} \big] \rmd s \\
  = {} & - \E \big[\dot{\lambda}^x_{t, \theta^\ast}\big] \sum_{i=1}^{d_g} z_{g,i} \partial_{\theta_{g,i}} \int_{-\infty}^{t-} g(t-s; \theta_g) \rmd s  \\
  = {} &  - z_{g,1} \frac{\nu(x; \theta_\nu^0)}{1-\eta^0}.
\end{align*}
Taking the derivative with respect to $\nu(x; \theta_\nu^0)/(1-\eta^0)$ on both sides of the above expression, we deduce
that $z_{g,1} = 0$ and $\sum_{i=1}^{d_\nu} z_{\nu, i} \partial_{\theta_{\nu, i}} \nu(x; \theta_\nu) = 0$, which, in view
of \hyperref[cond:C5]{[C5]-(ii)}, implies $z_{\nu, i} = 0$ for all $i \in \{1, \ldots, d_\nu\}$. From here, we have
\begin{equation*}
  0= \int_{-\infty}^{t-} \sum_{i=2}^{d_g} z_{g, i} \partial_{\theta_{g, i}} g(t-s; \theta_g) \rmd \dot{N}^x_s,
\end{equation*}
almost surely. Following the argument presented after \eqref{eq:consistency-proof-eq5} to deduce $f(t) = 0$ for all
$t \in \R_+$, we can similarly deduce that $\sum_{i=2}^{d_g} z_{g, i} \partial_{\theta_{g, i}} g(t; \theta_g) = 0$ for
all $t \in \R_+$, which, in view of \hyperref[cond:C5]{[C5]-(iii)}, implies $z_{g, i} = 0$ for all
$i \in \{2, \ldots, d_g\}$. This proves $\dot{\gamma}(x, \theta^\ast)$ is positive definite for all $x \in [0,1]$ and
hence so is $\Gamma(\theta^\ast)$.
\end{proof}
\section{Examples and numerical simulations}\label{sec:exampl-numer-simul}
\subsection{Examples of non-exponential $g$}\label{sec:exampl-non-expon}
In this section, we give examples of $\nu$ and non-exponential $g$ satisfying \ref{cond:C1}, \ref{cond:C2}, \hyperref[cond:C3]{[C3](p)}, \hyperref[cond:C4]{[C4](p)} and \ref{cond:C5}. Furthermore, we will verify, via numerical simulations, consistency and asymptotic normality of the MLE for three classes of non-exponential kernels: generalised Pareto, Gamma and Weibull. 

Note that \ref{cond:C1} and \hyperref[cond:C5]{[C5]-(i)} are regularity conditions on the model and can be easily verified. Let us focus on the conditions imposed on $\nu$ first, namely \ref{cond:C2}, \hyperref[cond:C4]{[C4](p)-(i)} and \hyperref[cond:C5]{[C5]-(ii)}. Upon inspection, these conditions are concerned with the continuity and boundedness of $\nu$ and its partial derivatives with respect to $\theta$, up to order 2. These conditions are straightforward to verify and are not overly restrictive for practical model applications. For example, the model considered in the simulation study by \cite{chen2013inference} has a quadratic baseline intensity function given by
\begin{equation}\label{eq:sim-stud-eq1}
	\nu(x; \theta_\nu) = e^{\theta_{\nu, 1}} + \left(e^{\theta_{\nu,2}} + e^{\theta_{\nu, 3}}\right)^2 \left(x - e^{\theta_{\nu, 2}}/\big(e^{\theta_{\theta_\nu, 2}} + e^{\theta_{\nu, 3}} \big)\right)^2. \nonumber
\end{equation}
By inspection, there exists $\underline{\nu}$, $\bar{\nu}$ and a compact $\Theta_\nu$ such that
$0 < \underline{\nu} \le \nu(x; \theta_\nu) \le \bar{\nu} < \infty$ for all
$(x, \theta_\nu) \in [0,1]\times \Theta_\nu$. Furthermore, for any $(x, \theta_\nu) \in [0,1]\times \Theta_\nu$ and
$j \in \{0, 1, 2\}$, $\vert \partial_\theta^{\otimes j} \nu(x; \theta_\nu) \vert$ is bounded and the families of
functions $\{\partial_{\theta_\nu} ^{\otimes j} \nu(x; \theta_\nu)\}_{\theta_\nu \in \Theta_\nu}$ and
$\{ \partial_{\theta_\nu}^{\otimes j} \nu(x; \theta_\nu)\}_{x\in [0,1]}$ are both equicontinuous. Finally, by computing
the first partial derivative of $\nu(x; \theta_\nu)$ with respect to $\theta_\nu$, \hyperref[cond:C5]{[C5]-(ii)} holds
trivially. Another suitable candidate of $\nu$ can be constructed via B-spline basis functions. Denote
$B^{n_b,k}_i(\cdot)$ the $i$th B-spline basis function of order $n_b$ and $k$ number of internal knots, the baseline
intensity function takes the form
\begin{equation}\label{eq:sim-stud-eq2}
  \nu(x; \theta_\nu) = \sum_{i=1}^{n_b + k}\theta_{\nu, i} B^{n_b, k}_i(x).
\end{equation}
While there are other families of functions that satisfy the conditions imposed on $\nu$, for the rest of this section,
we assume that the baseline intensity function takes the form as in \eqref{eq:sim-stud-eq2} and the aforementioned
conditions imposed on $\nu$ are all satisfied.

On the other hand, the generalised Pareto, Gamma and Weibull distributions all satisfy the conditions imposed on the
excitation kernel, namely \hyperref[cond:C3]{[C3](p)}, \hyperref[cond:C4]{[C4](p)-(ii)} and
\hyperref[cond:C5]{[C5](p)-(iii)}, with some restrictions on the parameters as a function of $p$. The density functions
for the generalised Pareto, Gamma and Weibull distributions are respectively
\begin{align*}
  & \tilde{g}_p(t; \alpha, \beta) \defeq \frac{1}{\beta}\Big(1 + \frac{\alpha t}{\beta} \Big)^{-(1/\alpha +1)},  & \alpha >0, \beta >0, t > 0;\\
  & \tilde{g}_g(t; \alpha, \beta) \defeq \frac{1}{\Gamma(\alpha) \beta^\alpha} t^{\alpha - 1} e^{-t/\beta},  & \alpha >0, \beta >0, t >0;  \\
  & \tilde{g}_w(t; \alpha, \beta) \defeq \frac{\alpha}{\beta}\Big(\frac{t}{\beta}\Big)^{\alpha-1} e^{-(t/\beta)^\alpha},  & \alpha >0, \beta >0, t >0; 
\end{align*}
where in all three cases, $\alpha$ denotes the shape parameter and $\beta$ the scale parameter. The generalised Pareto
density function takes the value $1/\beta$ at the origin $t=0$ and decays towards 0 at a polynomial rate. Relative to
the exponential kernel, $\tilde{g}_p$ decays at a much slower rate and has a heavier tail, thus making the generalised
Pareto kernel prefereable to the exponential kernel when modelling phenomena in which events in the past have longer
term effects on the current intensity. Note that both the exponential and generalised Pareto kernels decay
monotonically, hence, if we wish to model events where the peak excitation effect of an occurrence is not immediate but
rather at a delay, the Gamma or Weibull kernel with $\alpha > 1$ may provide a better fit.

Note that, the integrability conditions imposed under \hyperref[cond:C3]{[C3](p)} and \hyperref[cond:C4]{[C4](p)-(ii)}
may further restrict the parameter space $\tilde{\Theta}$ of the density function. Taking the generalised Pareto kernel
as an example, by computing the partial derivatives of $\tilde{g}_p(t; \alpha,\beta)$ with respect to $\alpha$ and
$\beta$ up to order 2 (the exact form of these partial derivatives are given in \hyperref[sec:Examples-of-g]{Appendix
  C}, a quick calculations shows that $\tilde{g}_p(t; \alpha, \beta)$ satisfies the integrability condition under
\hyperref[cond:C3]{[C3](p)} for $\alpha < 1/2$ and $\beta >0$. Following a similar argument, it is can be shown that
$\tilde{g}_p(t; \alpha, \beta)$ satisfies \hyperref[cond:C4]{[C4](p)-(ii)} for $\alpha < 1/3$ and $\beta > 0$. By taking
a similar approach, we can also deduce that when $\alpha > 1/p$ and $\beta >0$, both $\tilde{g}_g(t; \alpha, \beta)$ and
$\tilde{g}_w(t; \alpha, \beta)$ satisfie \hyperref[cond:C3]{[C3](p)} and \hyperref[cond:C4]{[C4](p)-(ii)}. The details
on this are relegated to \hyperref[sec:Examples-of-g]{Appendix C}.

\subsection{Numerical simulations}\label{sec:numer-simul}
In the previous subsection, we gave examples of $\nu$ and considered three families of $g$ such that the intensity
process satisfies the regularity conditions under which the MLEs are consistent and asymptotically normal. In this
subsection, we investigate the impact of sample size on these properties using simulations of three different models as
summarized in Table~\ref{tab:sim-study-model-sum}.. For each model, the baseline intensity function takes the form
\eqref{eq:sim-stud-eq2} with $n_b = 2$, $k = 1$, and
$(\theta_{\nu, 1}, \theta_{\nu, 2}, \theta_{\nu, 3}) = (5, 1.25, 2.5)$. The resulting shape of the baseline intensity is
motivated by phenomena that daily stock trading intensities tend to be higher during market open and market close, and
lower during middle of the day \cite{chen2013inference,stindl2018likelihood,kwan2022alternative}.

The first model assumes a generalised Pareto kernel $g(t; \theta_g) = \eta \tilde{g}_p(t; \alpha \beta)$ with
$(\eta, \alpha, \beta) = (0.5, 0.25, 0.75)$, the second model assumes a Gamma kernel
$g(t; \theta_g) = \eta \tilde{g}_g(t; \alpha, \beta)$ with $(\eta, \alpha, \beta) = (0.5, 2, 0.5)$ and the third model
assumes a Weibull kernel $g(t; \theta_g) = \eta \tilde{g}_w(t; \alpha, \beta)$ with
$(\eta, \alpha, \beta) = (0.5, 2, \Gamma(1.5)^{-1})$. 

\begin{table}[hbt]
  \centering
  \begin{adjustbox}{width=1\textwidth}
    \begin{tabular}{l l l }
      \toprule
      & \multicolumn{1}{c}{Baseline Intensity} & \multicolumn{1}{c}{Excitation Kernel}  \\
      \midrule
      Model 1 & $ \nu(x; \theta_\nu) = 5 B^{2,1}_1(x) + 1.25 B^{2,1}_2(x) + 2.5 B^{2,1}_3(x)$ & $g(t; \theta_g) = 0.5 \tilde{g}_p(t; 0.25, 0.75)$ \\
      Model 2 & $ \nu(x; \theta_\nu) = 5 B^{2,1}_1(x) + 1.25 B^{2,1}_2(x) + 2.5 B^{2,1}_3(x)$ & $g(t; \theta_g) = 0.5 \tilde{g}_g(t; 2, 0.5)$  \\
      Model 3 & $ \nu(x; \theta_\nu) = 5 B^{2,1}_1(x) + 1.25 B^{2,1}_2(x) + 2.5 B^{2,1}_3(x)$ & $g(t; \theta_g) = 0.5 \tilde{g}_w(t; 2, \Gamma(1.5)^{-1})$ \\
      \bottomrule
    \end{tabular}
  \end{adjustbox}
  \caption{Summary of the models considered in the simulation study.}\label{tab:sim-study-model-sum}
\end{table}

For each value of $n$ selected, we simulated $1000$ replicates of sample paths for the above models on the interval
$[0,1]$ with baseline intensity function $n\nu(\cdot; \theta_\nu)$ and excitation kernel $ng(\cdot; \theta_g)$ for
$g = \eta \tilde{g}_p$, $g = \eta \tilde{g}_g$, $g = \eta \tilde{g}_w$. For the three models we present results for
$n= 100, 400, 1600$. As we shall soon see, for the generalised Pareto kernel case, the rate of convergence of the MLEs
to their true parameter values and to normality is much slower compared to the case of the Gamma or Weibull kernel. For
this reason, we include additional results for $n = 6400$ for this model.

The simulation results are summarised in Tables~\ref{tab:sim-study-sum-results-1} –
\ref{tab:sim-study-sum-results-3}. In these tables, the first column gives the value of $n$, the second column gives the
mean number of events of the 1000 simulations, the third column gives the parameter, the fourth column gives the true
parameter value, the fifth column gives the mean of the 1000 estimates, the sixth column gives the standard error
(obtained as the standard deviation of the 1000 estimates divided by $\sqrt{n}$) and the last column gives the p-value
of the Kolnogorov-Smirnov (KS) test of normality of the 1000 estimates where the mean and standard deviation of the
theorised normal distribution are chosen as the empirical mean and standard deviation of the 1000 estimates given in the
fifth and sixth columns of the tables, respectively. Figures~\ref{fig:sim-study-qq-GP} – \ref{fig:sim-study-qq-WB}
present the QQ-plots, against normality, of the parameter estimates for Models 1 – 3 respectively.

From Tables~\ref{tab:sim-study-sum-results-1} – \ref{tab:sim-study-sum-results-3}, we observe that the mean estimates of
the parameters and their respective true values are close, relative to their standard errors, as $n$
increases. Furthermore, the empirical standard error shrinks by a factor of roughly $2$ each time $n$ quadruples, which
is as expected. Across the three specifications of kernel estimates of the baseline parameters tend to be biased upwards
to a similar level but the amount of bias is small relative to their standard errors and essentially disappears when
sample size $n=1600$ is reached. Also, across all three kernel specifications, the branching ratio $\eta$ tends to be
underestimated for sample sizes $n=100, 400$ and again this bias disappears when $n=1600$.

The final column of Tables Tables~\ref{tab:sim-study-sum-results-1} – \ref{tab:sim-study-sum-results-3} indicate that
convergence to normality is ultimately reached for all parameters. However, it is notable that the estimated of
parameters $\alpha$ and $\beta$ (which define the shape of the excitation kernels) require larger sample sizes to
achieve normality than is required for the baseline function and branching ration $\eta$. This is particularly true for
the generalised Pareto kernel which requires a larger sample size than for the other two cases to attain
normality. These observation are borne out in the QQ-plots. Figure~\ref{fig:sim-study-qq-GP} indicate that the shape
parameter estimate, $\hat{\alpha}$, for Model 1 is significantly positively skewed for $n = 100$, with slight and slow
improvements as $n$ is progressively quadrupled. Similar patterns are also present for the estimated shape parameter for
Models 2 and 3. Additionally, for the generalise Pareto kernel, $\hat{\alpha}=0$ for a large percentages of the
simulates.  Note that when $\alpha = 0$, the generalised Pareto distribution becomes the exponential distribution with
scale parameter equals to $\beta$. Thus, we can infer from Figure~\ref{fig:sim-study-qq-GP} that a significant
proportion of the fitted models suggests a Hawkes process with an exponential kernel, and that proportion decreases
inversely with $n$. The lack of events from small values of $n$ may contribute to this phenomenon, which when compounded
with the substantial edging effect from the generalised Pareto kernel, could result in a situation where the model has
insufficient information to justify the use a generalised Pareto kernel.

Interestingly, for all three models, and for all values of $n$, the mean estimates of $\eta$ are consistently below the
true value of $0.5$, and increase monotonically with $n$. However, this phenomenon does not necessarily suggest that the
models are consistently underestimating the number of offspring events, but it could be, again, a consequence of the
edging effect, since the simulations are right truncated, ergo not all offspring events are realised. Evident from
Table~\ref{tab:sim-study-sum-results-1}, for the p-values of the KS tests against normality of the estimated parameters
from Model 1 to be uniformly greater than 0.5, especially for $\hat{\eta}$, $\hat{\alpha}$, and $\hat{\beta}$, the model
demands a significantly larger value of $n$ compared to its alternatives. This is due to the rate of polynomial decay
from the generalised Pareto kernel, which resulted in a model with more a persistent and longer lasting excitation
effect. Naturally, such a model may require a much larger $n$ to conduct asymptotic inference reliably.

\begin{table}[hbt]
  \centering
  \begin{adjustbox}{width=1\textwidth}
    \begin{tabular}{>{\centering\arraybackslash}m{0.15\linewidth} c c c c c c c}
      \toprule
      Model 1 $\qquad$ (generalised Pareto): & Mean \({N^n_1}\)&Par.&True&Mean Est.&Emp. S.E. &  p-value (KS test)\\\hline
		
		\multirow{5}{*}{$n = 100$}&	\multirow{5}{*}{\( 581 \)}	&$\theta_{\nu,1}$ & 5.00 & 5.684 & 1.264 & 4.540 e-6 \\ 
		&&$\theta_{\nu,2}$& 1.25 & 1.342 & 1.411 & 0.660 \\
		& &$\theta_{\nu,3}$& 2.50 & 2.835 & 0.903 & 1.518 e-3\\
		& &$\eta$& 0.50 & 0.450 & 0.140 & 1.885 e-5 \\
		& &$\alpha$	&0.25 & 0.497 & 1.337 & 0.000  \\
		& &$\beta$	& 0.75 & 0.791 & 0.706 & 0.000  \\ \hline
		
		\multirow{5}{*}{$n = 400$}&	\multirow{5}{*}{\( 2329 \)}	& $\theta_{\nu, 1}$ & 5.00 & 5.207 & 0.653 & 0.056 \\
		& &$\theta_{\nu, 2}$ & 1.25 & 1.294 & 0.638 & 0.572 \\
		& &$\theta_{\nu, 3}$ & 2.50 & 2.609 & 0.446 & 0.070 \\
		& &$\eta$ &0.50 & 0.480 & 0.064 & 4.402 e-3 \\
		& &$\alpha$ & 0.25 & 0.239 & 0.410 & 0.000 \\
		& &$\beta$ & 0.75 & 0.740 & 0.166 & 1.879 e-4 \\ \hline
		
		\multirow{5}{*}{$n = 1600$}&	\multirow{5}{*}{\( 9328 \)}	&$\theta_{\nu, 1}$ & 5.00 & 5.066 & 0.375 & 0.223 \\ 
		& &$\theta_{\nu, 2}$ & 1.25 & 1.263 & 0.275 & 0.851 \\
		& &$\theta_{\nu, 3}$ & 2.50 & 2.529 & 0.223 & 0.900 \\
		& &$\eta$ &  0.50 & 0.494 & 0.033 & 2.671 e-3 \\
		& &$\alpha$ & 0.25 & 0.242 & 0.222 & 1.110 e-16 \\
		& &$\beta$ & 0.75 & 0.749 & 0.083 & 0.168 \\ \hline
		
		\multirow{5}{*}{$n = 6400$}&	\multirow{5}{*}{\( 37317 \)}	&$\theta_{\nu,1}$ & 5.00 & 5.034 & 0.188 & 0.638 \\
		&&$\theta_{\nu,2}$&1.25 & 1.260 & 0.136 & 0.851 \\
		& &$\theta_{\nu,3}$& 2.50 & 2.514 & 0.112 & 0.678 \\
		& &$\eta$& 0.50 & 0.496 & 0.017 & 0.157 \\
		& &$\alpha$	&0.25 & 0.251 & 0.119 & 0.486 \\
		& &$\beta$	& 0.75 & 0.749 & 0.040 & 0.412 \\ \hline
    \end{tabular}
  \end{adjustbox}
  \caption{The results of estimating the parameters using the ML method for Model 1.}\label{tab:sim-study-sum-results-1}
\end{table}

\begin{figure}[h!]
  \centering \includegraphics[width=\linewidth]{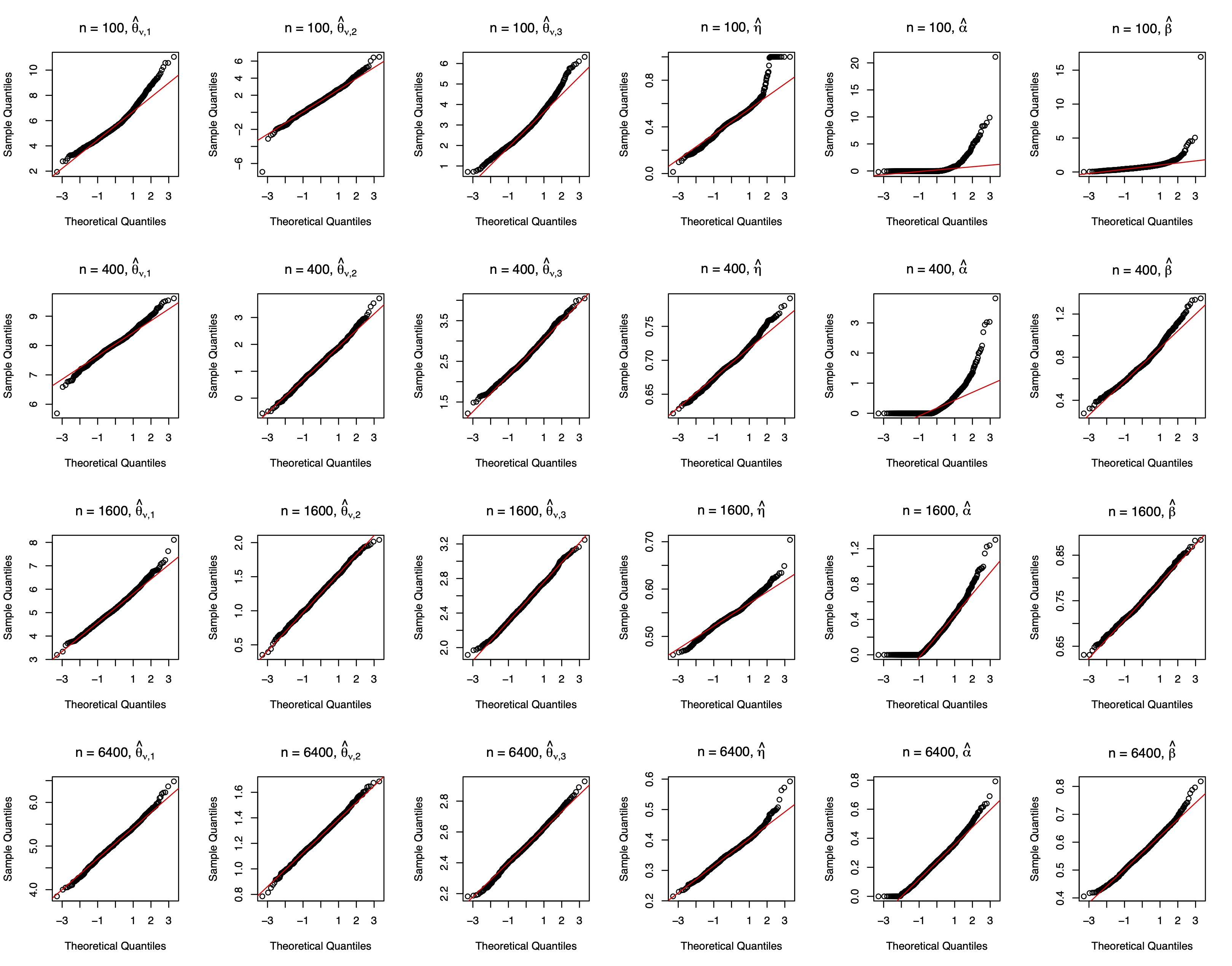}
  \caption{Normal QQ-plots of the estimated parameters for Model 1. The rows corresponds to $n$ = 100, 400, 1600 and
    6400 from top to bottom, and the columns corresponds to the parameters $\hat{\theta}_{\nu, 1}$,
    $\hat{\theta}_{\nu, 2}$, $\hat{\theta}_{\nu, 3}$, $\hat{\eta}$, $\hat{\alpha}$ and
    $\hat{\beta}$.}\label{fig:sim-study-qq-GP}
\end{figure}

\begin{table}[hbt]
	\centering
	\begin{adjustbox}{width=1\textwidth}
	\begin{tabular}{>{\centering\arraybackslash}m{0.15\linewidth} c c c c c c c}
		\toprule
		Model 2 (Gamma): & Mean \({N^n_1}\)&Par.&True&Mean Est.&Emp. S.E. &  p-value (KS test)\\\hline
		
		\multirow{5}{*}{$n = 100$}&	\multirow{5}{*}{\( 576 \)}	&$\theta_{\nu,1}$ & 5 & 5.737 & 1.288 & 0.003 \\
		&&$\theta_{\nu,2}$& 1.25 &1.484 & 1.562 & 0.007 \\ 
		& &$\theta_{\nu,3}$& 2.5 & 2.933 & 1.070 & 0.000  \\ 
		& &$\eta$& 0.5 &0.417 & 0.161 & 0.000 \\ 
		& &$\alpha$	&2 & 4.768 & 6.797 & 0.000 \\ 
		& &$\beta$	& 0.5 & 0.564 & 1.267 & 0.000 \\ \hline
		
	%	\multirow{5}{*}{$n = 200$}&	\multirow{5}{*}{\( 1157 \)}	&$\theta_{\nu,1}$ & 5 & 5.384 & 0.914 & 0.023 \\
		%&&$\theta_{\nu,2}$& 1.25 &1.368 & 0.864 & 0.247 \\ 
		%& &$\theta_{\nu,3}$& 5 & 2.709 & 0.596 & 0.020 \\
		%& &$\eta$& 0.5 &0.456 & 0.083 & 0.482 \\
		%& &$\alpha$	&2 & 2.852 & 2.433 & 0.000 \\
		%& &$\beta$	& 0.5 & 0.495 & 0.408 & 0.000 \\ \hline
		
		\multirow{5}{*}{$n = 400$}&	\multirow{5}{*}{\( 2331 \)}	&$\theta_{\nu,1}$ & 5 & 5.205 & 0.680 & 0.322 \\
		&&$\theta_{\nu,2}$&1.25 & 1.332 & 0.585 & 0.863 \\
		& &$\theta_{\nu,3}$& 2.5 & 2.604 & 0.431 & 0.159 \\
		& &$\eta$& 0.5 & 0.478 & 0.059 & 0.170 \\
		& &$\alpha$	&2 & 2.302 & 0.813 & 0.000 \\
		& &$\beta$	& 0.5 & 0.491 & 0.206 & 0.000 \\ \hline
		
		%\multirow{5}{*}{$n = 800 $}&	\multirow{5}{*}{\( 4689 \)}	&$\theta_{\nu,1}$ & 5 & 5.072 & 0.523 & 0.378 \\
		%&&$\theta_{\nu,2}$& 1.25 &1.279 & 0.426 & 0.455 \\
		%& &$\theta_{\nu,3}$& 2.5 & 2.527 & 0.316 & 0.689 \\
		%& &$\eta$& 0.5 & 0.491 & 0.042 & 0.712 \\
		%& &$\alpha$	&2 & 2.107 & 0.591 & 0.010 \\
		%& &$\beta$	& 0.5 & 0.517 & 0.170 & 0.506 \\ \hline
		
		\multirow{5}{*}{$n = 1600$}&	\multirow{5}{*}{\( 9323 \)}	&$\theta_{\nu,1}$ & 5 & 5.022 & 0.340 & 0.839 \\
		&&$\theta_{\nu,2}$& 1.25 &1.289 & 0.249 & 0.991 \\
		& &$\theta_{\nu,3}$& 2.5 &2.552 & 0.187 & 0.998 \\
		& &$\eta$& 0.5 & 0.494 & 0.027 & 0.780 \\
		& &$\alpha$	&2 & 2.069 & 0.282 & 0.517 \\
		& &$\beta$	& 0.5 & 0.496 & 0.095 & 0.842 \\ \hline
		
	\end{tabular}
      \end{adjustbox}
\caption{The results of estimating the parameters using the ML method for Model 2.}\label{tab:sim-study-sum-results-2}
\end{table}

\begin{figure}[h!]
  \centering \includegraphics[width=\linewidth]{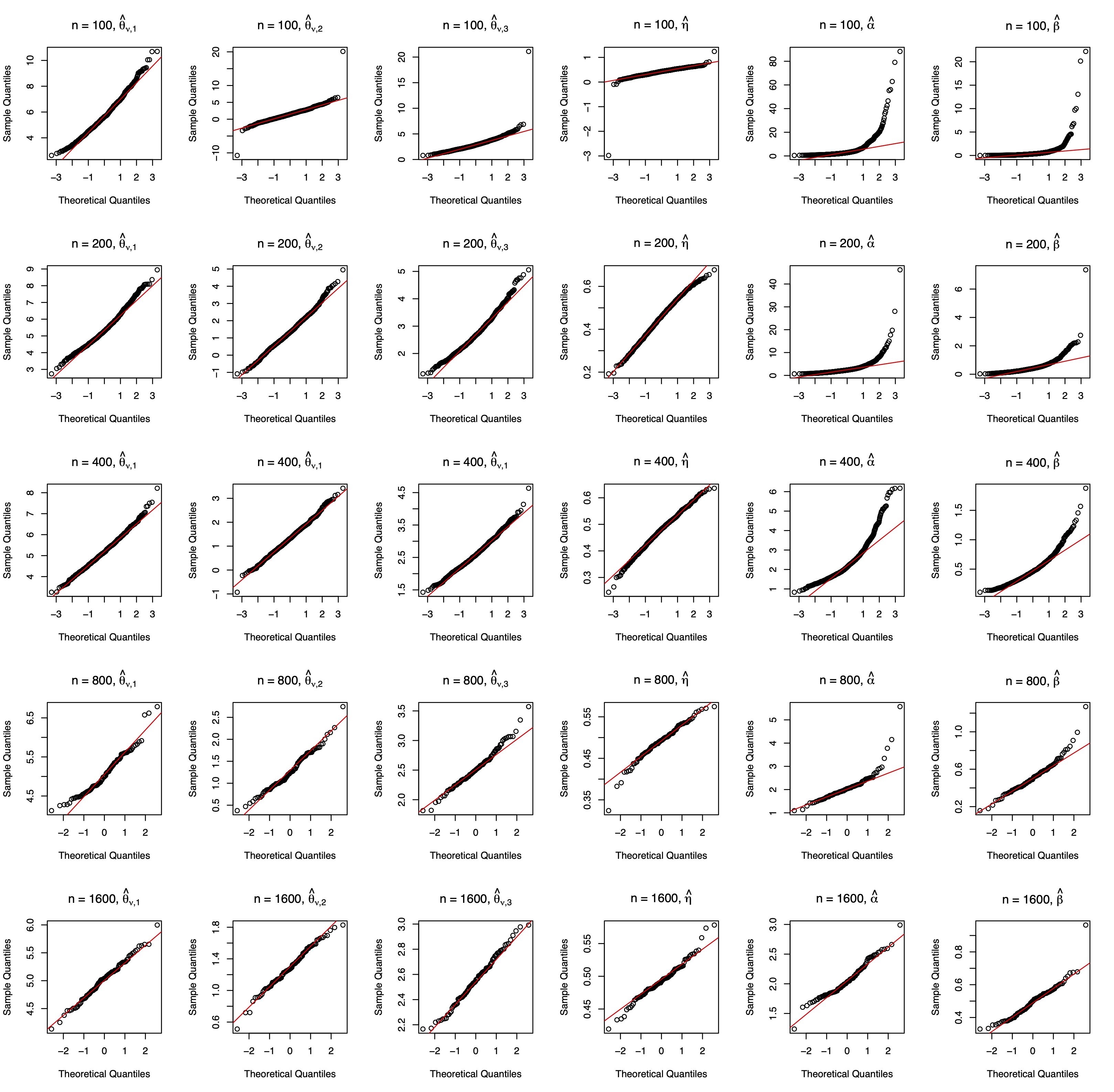}
  \caption{Normal QQ-plots of the estimated parameters for Model 2. The rows correspond to $n$ = 100, 200, 400, 800 and
    1600 from top to bottom, and the columns correspond to the parameters $\hat{\theta}_{\nu, 1}$,
    $\hat{\theta}_{\nu, 2}$, $\hat{\theta}_{\nu, 3}$, $\hat{\eta}$, $\hat{\alpha}$ and
    $\hat{\beta}$.}\label{fig:sim-study-qq-GM}
\end{figure}

\begin{table}[hbt]
  \centering
  \begin{adjustbox}{width=1\textwidth}
    \begin{tabular}{>{\centering\arraybackslash}m{0.15\linewidth} c c c c c c c}
      \toprule
      Model 3 $\qquad$ (Weibull): & Mean \({N^n_1}\)&Par.&True&Mean Est.&Emp. S.E. &  p-value (KS test)\\\hline
		
      \multirow{5}{*}{$n = 100$}&	\multirow{5}{*}{\( 581 \)}	&$\theta_{\nu,1}$ & 5& 5.737 & 1.306 & 0.010 \\
                                  &&$\theta_{\nu,2}$& 1.25 & 1.510 & 1.488 & 0.143 \\
                                  & &$\theta_{\nu,3}$& 2.5 & 2.900 & 0.892 & 0.013 \\
                                  & &$\eta$& 0.5 & 0.421 & 0.126 & 0.352 \\
                                  & &$\alpha$	& 2 & 1.503 & 1.328 & 0.000 \\
                                  & &$\beta$	& 1.128 & 1.103 & 1.325 & 0.000 \\ \hline
		
		%\multirow{5}{*}{$n = 200$}&	\multirow{5}{*}{\( 1161 \)}	&$\theta_{\nu,1}$ & 5 & 5.496 & 1.020 & 0.005 \\
		%&&$\theta_{\nu,2}$&1.25 & 1.397 & 0.870 & 0.526 \\
		%& &$\theta_{\nu,3}$& 2.5 & 2.781 & 0.658 & 0.008 \\
		%& &$\eta$& 0.5 & 0.446 & 0.091 & 0.064 \\
		%& &$\alpha$	&2 &1.153 & 0.466 & 0.000 \\ 
		%& &$\beta$	& 1.128 & 0.966 & 0.387 & 0.000 \\\hline
		
		\multirow{5}{*}{$n = 400$}&	\multirow{5}{*}{\( 2325 \)}	&$\theta_{\nu,1}$ & 5 & 5.251 & 0.690 & 0.398 \\
		&&$\theta_{\nu,2}$& 1.25 & 1.366 & 0.600 & 0.266 \\ 
		& &$\theta_{\nu,3}$& 2.5 & 2.614 & 0.427 & 0.263 \\
		& &$\eta$& 0.5 &0.471 & 0.061 & 0.853 \\ 
		& &$\alpha$	&2 &1.056 & 0.171 & 0.000 \\ 
		& &$\beta$	& 1.128 & 0.963 & 0.234 & 0.001 \\  \hline
		
		%\multirow{5}{*}{$n = 800$}&	\multirow{5}{*}{\( 4671 \)}	&$\theta_{\nu,1}$ & 5 & 5.102 & 0.482 & 0.768 \\
		%&&$\theta_{\nu,2}$& 1.25 &1.277 & 0.410 & 0.926 \\ 
		%& &$\theta_{\nu,3}$& 2.5 &2.547 & 0.300 & 0.029 \\
		%& &$\eta$& 0.5 & 0.491 & 0.039 & 0.987 \\
		%& &$\alpha$	&2& 2.092 & 0.289 & 0.005 \\ 
		%& &$\beta$	& 1.128 & 1.129 & 0.091 & 0.415 \\ \hline
		
		\multirow{5}{*}{$n = 1600$}&	\multirow{5}{*}{\( 9332 \)}	&$\theta_{\nu,1}$ & 5 & 5.043 & 0.331 & 0.479 \\ 
		&&$\theta_{\nu,2}$& 1.25 &1.251 & 0.254 & 0.905 \\ 
		& &$\theta_{\nu,3}$& 2.5 & 2.523 & 0.199 & 0.625 \\
		& &$\eta$& 0.5 & 0.496 & 0.026 & 0.772 \\
		& &$\alpha$	&2& 2.028 & 0.187 & 0.394 \\ 
		& &$\beta$	& 1.128 & 1.128 & 0.066 & 0.618 \\ \hline
    \end{tabular}
  \end{adjustbox}
  \caption{The results of estimating the parameters using the ML method for Model 3.}\label{tab:sim-study-sum-results-3}
\end{table}

\begin{figure}[h!]
  \centering \includegraphics[width=\linewidth]{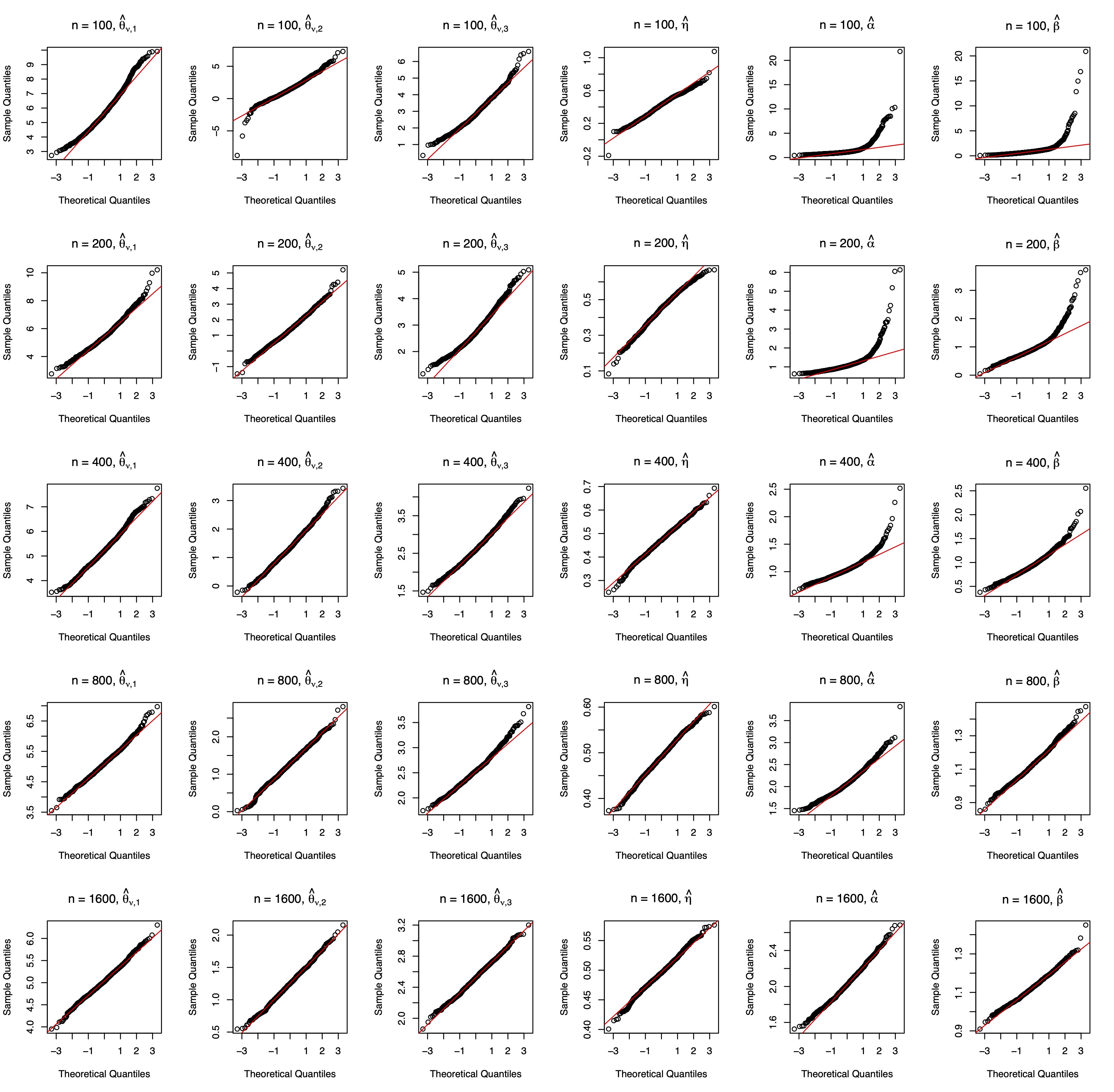}
  \caption{Normal QQ-plots of the estimated parameters for Model 3. The rows correspond to $n$ = 100, 200, 400, 800 and
    1600 from top to bottom, and the columns correspond to the parameters $\hat{\theta}_{\nu, 1}$,
    $\hat{\theta}_{\nu, 2}$, $\hat{\theta}_{\nu, 3}$, $\hat{\eta}$, $\hat{\alpha}$ and
    $\hat{\beta}$.}\label{fig:sim-study-qq-WB}
\end{figure}
\section{Conclusion}\label{sec:conclusion}

\begin{appendix}\label{appendix}

  \section{Preparation of proofs}\label{sec:preperation-proofs}
  Throughout these appendices, we write $K$ for a finite and positive constant, and $K_p$ for some finite and positive
  constant that depends only on $p$, both of which might change from line to line, since we are only interested in their
  bounding qualities.
  \begin{lemma}\label{lm:prep-proofs-lm1}
    For $p \in \N$, under \ref{cond:C1}, \ref{cond:C2} and \hyperref[cond:C3]{[C3](p)}, for any $j \in \{0,1,2\}$ and
    $n \in \N$, we have
    \begin{enumerate}
    \item [{(i)}]
      $\sup_{t \in [0, n]} \sup_{\theta \in \Theta} \sum_{j=0}^2 \Vert \partial_\theta^{\otimes j} \lambda^n_{t, \theta}
      \Vert_p < \infty$;
    \item [{(ii)}]
      $\sup_{t \in [0, b_n]} \sup_{x \in [0,1]} \sup_{\theta \in \Theta} \sum_{j=0}^2 \Vert \partial_\theta^{\otimes
        j}\lambda^{x,n}_{t, \theta} \Vert_p < \infty$;
    \item[{(iii)}]
      $\sup_{t \in \R_+} \sup_{x \in [0,1]} \sup_{\theta \in \Theta} \sum_{j=0}^2 \Vert \partial_\theta^{\otimes j}
      \lambda^{x,n,c}_{t, \theta} \Vert_p < \infty$;
    \item[{(iv)}]
      $\sup_{t \in \R} \sup_{x \in [0,1]} \sup_{\theta \in \Theta} \sum_{j=0}^2 \Vert \partial_\theta^{\otimes j}
      \dot{\lambda}^{x,n,c}_{t, \theta} \Vert_p < \infty$.
    \end{enumerate}
  \end{lemma}
\subsection*{Proof of Lemma~\ref{lm:prep-proofs-lm1}}
While this proof follows a similar structure as the proofs of Lemma~S.1 of \cite{kwan2022alternative} and Lemma~8 of
\cite{kwan2024ergodic}, wefirst give details required to show $(i)$. Let us consider the case when $j = 0$, and
$\theta = \theta^0$ first. We will prove by induction on $p$ that
$\sup_{t \in [0,n]} \Vert \lambda^n_{t, \theta^0} \Vert_p < \infty$. For the initialising case when $p = 1$, we have
\begin{align*}
  \Vert \lambda^n_{t, \theta^0} \Vert= {}  \nu(\frac{t}{n}; \theta^0) + \int_{0}^{t^-} g(t-s; \theta^0) \mathbb{E} [\lambda^n_{s, \theta^0} ] \rmd s 	\le {}  \bar{\nu}+ \bar{\eta} \sup_{t \in [0,n]} \mathbb{E} [ \lambda^n_{t, \theta^0}].
\end{align*}
From there, it is straightforward to see that
$\sup_{t \in [0,n]} \Vert \lambda^n_{t, \theta^0} \Vert \le \bar{\nu}/(1-\bar{\eta}) < \infty$. For $p >1$, from the
inequality that $ (x + y)^{p} \le (1+\tilde{\epsilon})^{p-1} x^p + (1+\tilde{\epsilon}^{-1})^{p-1}y^p$ for
$x,y,\tilde{\epsilon} >0$, we can bound $\mathbb{E}[ \vert \lambda^n_{t, \theta^0} \vert^p]$ by
\[ (1+\tilde{\epsilon})^{p-1}\bar{\nu} ^p + (1+\tilde{\epsilon}^{-1})^{p-1} \mathbb{E} \Big[\Big(\int_{0}^{t^-}
  g(t-s;\theta^0) \rmd N^n_s \Big)^p \Big]. \] The first term is clearly bounded under~\hyperref[cond:C1]{[C1]}. Using
the same inequality to get the above bound, $\mathbb{E} [(\int_{0}^{t^-} g(t-s;\theta^0) \rmd N^n_s )^p]$ is bounded by
\begin{equation*}
  (1+\tilde{\epsilon})^{p-1} \mathbb{E} \Big[\Big(\int_{0}^{t^-} g(t-s;\theta^0) \rmd \bar{N}^n_s \Big)^p \Big] +  (1+\tilde{\epsilon}^{-1})^{p-1} \mathbb{E} \Big[\Big(\int_{0}^{t^-} g(t-s;\theta^0) \lambda^n_{s, \theta^0} \rmd s \Big)^p\Big].
\end{equation*}
An application of Lemma A.2 of~\cite{clinet2017statistical} gives the following bound for
$ \mathbb{E} [(\int_{0}^{t^-} g(t-s;\theta^0) \rmd \bar{N}^n_s )^p] $:
\begin{equation*}
  K_p \Big\{ \mathbb{E} \Big[\int_{0}^{t^-} g(t-s;\theta^0)^{p} \lambda^n_{s, \theta^0} \rmd s \Big] + \mathbb{E} \Big[ \Big( \int_{0}^{t^-} g(t-s;\theta^0)^{2} \lambda^n_{s, \theta^0} \rmd s \Big)^{p/2} \Big]  \Big\}.
\end{equation*}
Since $g(\cdot; \theta^0)^p$ is integrable and $\sup_{t \in [0,n]} \Vert \lambda^n_{t, \theta^0} \Vert < \infty$, we
deduce that the first term above is bounded. Recall from H\"{o}lder's inequality that
$(\int f g)^k \le (\int f^k g)(\int g)^{k-1}$ for $k >1$ and non-negative functions $f$, $g$, we have the following
bound for $\mathbb{E} [ ( \int_{}^{t^-} g(t-s;\theta^0)^{i} \lambda^n_{s, \theta^0} \rmd s )^{j} ] $:
\begin{align*}
  {} & \mathbb{E} \bigg[ \Big( \int_{0}^{t^-} g(t-s;\theta^0)^{i} \big(\lambda^n_{s, \theta^0} \big)^{j}  \rmd s \Big) \Big( \int_{0}^{t^-} g(t-s;\theta^0)^i \rmd s \Big)^{j-1}  \bigg] \\
  \le {} & \sup_{t \in [0,n]} \mathbb{E} \Big[\big(\lambda^n_{t, \theta^0} \big)^j \Big] \Big(\int_{0}^{t^-} g(t-s; \theta^0)^i \rmd s \Big)^{j}.
\end{align*}
For $i \in [1,p]$ and $j \in [1,p)$, the above bound is finite under the assumed conditions. To this end, we have shown
that
\[ \mathbb{E} \big[\vert \lambda^n_{t, \theta^0} \vert^p \big] \le (1+\tilde{\epsilon})^{p-1} \bar{\nu}^p + \big( (1 +
  \tilde{\epsilon}) (1+\tilde{\epsilon}^{-1})\big)^{p-1} K_q + (1 + \tilde{\epsilon}^{-1})^{2p - 2} \bar{\eta}^p \sup_{t
    \in [0, n]} \mathbb{E} \big[\vert \lambda^n_{t, \theta^0} \vert^p \big]. \] By taking the supremum in $t$ over
$[0,n]$ and rearranging, it follows that $\sup_{t \in [0, n]} \Vert \lambda^n_{t, \theta^0} \Vert_p$ is bounded.
	
To generalise the result such that it holds for any $\theta \in \Theta$. We will do so by induction on $p$ again. For
the initialising case when $p = 1$, we have
\begin{align*}
  \big\Vert \lambda^n_{t, \theta}  \Vert \le {}  \bar{\nu} + \mathbb{E} \bigg[  \int_{0}^{t^-} g(t-s; \theta) \lambda_{s, \theta^0} \rmd s \bigg]
  \le {}  \bar{\nu} + \bar{\eta}  \sup_{t \in [0,n] } \mathbb{E} [ \lambda^n_{t, \theta^0}] < \infty.
\end{align*}
For $p >1$, by an argument similar to before, $\mathbb{E}[ (\lambda^n_{t, \theta})^p]$ is bounded by
\[ (1+\tilde{\epsilon})^{p-1} \bar{\nu}^p + (1+\tilde{\epsilon}^{-1})^{p-1} \mathbb{E} \Big[\Big( \int_{0}^{t^-}
  g(t-s;\theta) \rmd N^n_s \Big)^p\Big] \] and $\mathbb{E} [( \int_{0}^{t^-} g(t-s;\theta) \rmd N^n_s )^p]$ is bounded
by
\begin{align}
  K_p \bigg\{ & \mathbb{E} \Big[ \int_{0}^{t^-} g(t-s; \theta)^p \lambda^n_{s, \theta^0} \rmd s\Big] + \mathbb{E} \Big[ \Big( \int_{0}^{t^-} g(t-s; \theta)^{2} \lambda^n_{s, \theta^0} \rmd s \Big)^{p/2} \Big]  \nonumber \\
              & + \mathbb{E} \Big[\Big( \int_{0}^{t^-} g(t-s; \theta) \lambda^n_{s, \theta^0} \rmd s \Big)^p \Big] \bigg\}. \label{eq:Chap_4_step_1_eq1}
\end{align}
By H\"{o}lder's inequality again, we have for $i, j \ge 1$,
\begin{align}
  {}  \mathbb{E} \bigg[ \Big( \int_{0}^{t^-} g(t-s; \theta)^i \lambda^n_{s, \theta^0} \rmd s \Big)^j \bigg]  
  \le {} & \mathbb{E} \bigg[ \Big( \int_{0}^{t^-} g(t-s; \theta)^i \big( \lambda^n_{s, \theta^0}\big)^j \rmd s \Big) \Big( \int_{0}^{t^-} g(t-s; \theta)^i \rmd s \Big)^{j-1} \bigg]  \nonumber \\
  \le {} &  \Big( \int_{0}^{t^-} g(t-s; \theta)^i \rmd s \Big)^{j} \sup_{t \in [0,n] } \mathbb{E}\big[ (\lambda^n_{t, \theta^0})^p  \big]. \label{eq:Chap_4_Holder_ineq_eq1}
\end{align}
Under the assumed conditions, induction hypothesis and the above inequality, the first two terms of
\eqref{eq:Chap_4_step_1_eq1} are trivially bounded, and the last term of \eqref{eq:Chap_4_step_1_eq1} is bounded by
$\bar{\eta}^p \mathbb{E}[ (\lambda^n_{t, \theta^0})^p ]$. Collecting these derived bounds and taking the supremum in $t$
and $\theta$ over $[0,n]$ and $\Theta$ respectively, this shows that
$\sup_{t \in \mathbb{R}} \sup_{\theta \in {\Theta}} \Vert \lambda^n_{t, \theta} \Vert_p$ is bounded. Finally, by the
boundedness of $\vert \partial_\theta^{\otimes j} \nu(x; \theta)\vert$ and integrability of
$\vert \partial_\theta^{\otimes j} g(t; \theta) \vert^p$, and a straightforward adaptation of the proof above, it can be
shown that
$\sup_{t\in [0,n]} \sup_{\theta \in \Theta} \Vert \partial_\theta^{\otimes j} \lambda^n_{t, \theta} \Vert_p < \infty$
for $j \in \{1, 2\}$.
\begin{remark}
  \label{remark:prep-proofs-remark1}
  The above lemma holds for non-integer $p$ as well since the BDF inequality holds for all $p \ge 1$. However, for
  $p \in (1, 2)$, due to the quadratic variation term in the BDG inequality, we require $g$ to be square integrable in
  order for the intensity process to be bounded in $\mathbb{L}_p$. For $p \ge 2$, the above proof can be easily adapted
  to achieve the desired result.
\end{remark}

\begin{lemma}\label{lm:prep-proofs-lm2}
  Under \ref{cond:C1}, \ref{cond:C2} and \hyperref[cond:C3]{[C3](p)}, for any $j \in \{0,1,2\}$, there exists $\gamma_n \rightarrow \infty$, $b_n \rightarrow
  \infty$ with $\gamma_n = o(b_n)$ such that
  \begin{enumerate}
  \item [{(i)}] $\sup_{t \in [\gamma_n, b_n]} \sup_{x \in [0,1]} \sup_{\theta \in \Theta} \Vert \partial_\theta^{\otimes j} \lambda^{x,n}_{t, \theta} -
    \partial_\theta^{\otimes j} \lambda^{x,n,c}_{t, \theta} \Vert_p = o(1)$;
    \item[{(ii)}] $\sup_{t \in [\gamma_n, b_n]} \sup_{x \in [0,1]} \sup_{\theta \in \Theta} \Vert \partial_\theta^{\otimes j} \lambda^{x,n,c}_{t, \theta} -
      \partial_\theta^{\otimes j} \dot{\lambda}^{x,n,c}_{t, \theta} \Vert_p = o(1)$.
  \end{enumerate}
\end{lemma}
\subsection*{Proof of Lemma~\ref{lm:prep-proofs-lm2}}
We only give details to show (i) since (ii) follow via a straightforward adaptation of the the proof of (i). Let us
consider the case when $p = 1$, $\theta = \theta^0$ and $j = 0$ first. By Poisson embedding and Fubini's theorem,
$\E [ \vert \lambda^{x,n}_{t, \theta^0} - \lambda^{x,n,c}_{t, \theta^0} \vert ]$ is bounded by
\begin{align*}
  & \big\vert \nu(x + \frac{t}{n}; \theta^0) - \nu(x; \theta^0) \big\vert + \int_0^{nx-} g(t + nx -s; \theta^0) \E\big[ \lambda^n_{s, \theta^0} \big] \rmd s \\
  + {} & \int_0^{t-} g(t-s; \theta^0) \E \big[ \big\vert \lambda^{x,n}_{s, \theta^0} - \lambda^{x,n,c}_{s, \theta^0} \big\vert \big] \rmd s.
\end{align*}
Under \hyperref[cond:C2]{[C2]}, $\nu(x; \theta)$ is continuous in $x$, and since
$\frac{t}{n} \le \frac{1}{a_n}$, the first term above can be made arbitrarily small by increasing $n$. Let
$G(t; \theta) = \int_0^t g(s; \theta) \rmd s$, by Lemma~\ref{lm:prep-proofs-lm1},
$\sup_{t \in [0,n]} \E [\lambda^n_{t, \theta^0}]$ is bounded and since $t \in [\gamma_n, b_n]$, the second term above is
bounded by
\begin{equation*}
  K \int_0^{nx} g(t + nx - s;\theta^0) \rmd s \le K \big( G(b_n + n; \theta^0) - G(\gamma_n; \theta^0) \big)
\end{equation*}
which converges to 0 as $n$ tends to infinity. Denote $ (a \ast b)(t) = \int_0^t a(t-s) b(s) \rmd s$ to be the
convolution between $a(\cdot)$ and $b(\cdot)$, $a^{\ast(i)}(\cdot)$ the $i$-th convolution of $a(\cdot)$ with itself and
$f(t; \theta) = \E [ \vert \lambda^{x,n}_{t, \theta} - \lambda^{x, n, c}_{t, \theta} \vert ]$, thus far, we have shown
that for any $\epsilon >0$, there exists an $N$ large enough such that for all $n > N$ and $t \in [\gamma_n, b_n]$, we
have
\begin{align}
  f(t ;\theta^0)  \le {} & \epsilon + \big(g (\cdot; \theta^0) \ast f(\cdot ;\theta^0) \big) (t) \label{eq:Chap_4_convol_eq1} \nonumber \\
  \le {} & \epsilon + \big(\tilde{\epsilon} \ast g(\cdot; \theta^0) \big) (t) + \big( g^{\ast(2)} (\cdot; \theta^0) \ast f(\cdot ; \theta^0) \big) (t) \nonumber \\
  \le {} & \epsilon\Big( 1 + \sum_{i=1}^k G^{\ast (i)}(t; \theta^0) \Big) + \big( g^{\ast (k+1)} (\cdot; \theta^0) \ast f(\cdot; \theta^0) \big) (t). \nonumber
\end{align}
Noting that $G^{\ast(i)}(t; \theta^0) \le \bar{\eta}^i$, we have
$\epsilon ( 1 + \sum_{i=1}^k G^{\ast (i)}(t; \theta^0) ) \le \frac{\epsilon}{1-\bar{\eta}}$. Furthermore, recall that
$g(\cdot; \theta) = \eta \tilde{g}(\cdot; \tilde{\theta})$ where $\tilde{g}(\cdot; \tilde{\theta})$ is the density
function of a positive random variable. Denote $\tilde{G}(t; \theta) = \int_0^t \tilde{g}(s; \tilde{\theta}) \rmd s$, by
Lemma~\ref{lm:prep-proofs-lm1} again, we can write
\begin{align*}
  \big( g^{\ast (k +1)} (\cdot; \theta^0) \ast f(\cdot ; \theta^0) \big) (t) = {} &  \int_0^t g^{\ast(k + 1)}(t-s; \theta^0) f(s; \theta^0) \rmd s \\
  \le {} & K \bar{\eta} ^{k+1} \tilde{G}^{\ast(k+1)}(t; \tilde{\theta}^0).
\end{align*}
Since $\tilde{G}^{\ast (k+1)}(t; \tilde{\theta^0}) \to 1$ as $n \to \infty$, thus for any $\tilde{\epsilon} > 0$, there
exists $K_1 > 0$ large enough such that for all $k > K_1$,
$\big( g^{\ast(k+1)} (\cdot; \theta^0) \ast f(\cdot; \theta^0) \big) (t) < \frac{\tilde{\epsilon}}{2}$. Collecting the
derived bounds from above, and let $\epsilon = \frac{(1 - \eta^0) \tilde{\epsilon}}{2}$, we have shown that there exist
$N$ large enough such that
$\sup_{t \in [\gamma_n, b_n]} \E [ \vert \lambda^{x,n}_{t, \theta^0} - \lambda^{x,n,c}_{t, \theta^0} \vert ] <
\tilde{\epsilon}$ whenever $n > N$.
	
To get a similar estimate for all $\theta \in \Theta$, by following the steps detailed above, it is straightforward to
see that for any $\epsilon >0$, there exist an $N$ large enough such that for all $n >N$, we have
\begin{equation*}
  f(t; \theta) \le \epsilon + \big(\ g(\cdot; \theta) \ast f(\cdot; \theta^0) \big) (t).
\end{equation*}
Note that there exists $\tilde{\gamma}_n \to \infty$ with $n$ such that $\tilde{\gamma}_n = o(\gamma_n)$ and
$\sup_{t \in [\tilde{\gamma}_n, b_n]} \E [\vert \lambda^{x,n}_{t, \theta^0} - \lambda^{x, n, c}_{t, \theta^0} \vert] =
o(1)$. In light of this, we have
\begin{align*}
  f(t; \theta) \le {} & \epsilon + \int_0^{t-} g(t-s; \theta) f(s; \theta^0) \rmd s \\
  = {} &  \epsilon + \int_0^{\tilde{\gamma}_n-} g(t-s; \theta) f(s; \theta^0) \rmd s + \int_{\tilde{\gamma}_n}^{t-} g(t-s; \theta) f(s; \theta^0) \rmd s\\
  < {} & \epsilon + K \big( G(t; \theta) - G(t - \tilde{\gamma}_n; \theta)\big) + \eta \sup_{t \in [\tilde{\gamma}_n, b_n]} \E \big[ \big\vert \lambda^{x,n}_{t, \theta^0} - \lambda^{x, n, c}_{t, \theta^0} \big \vert\big].
\end{align*}
Since $t \in [\gamma_n, b_n]$ and $\tilde{\gamma}_n = o(\gamma_n)$, there exists an $N$ large enough such that for all
$n > N$, $K (G(t; \theta) - G(t - \tilde{\gamma}_n; \theta) ) < \epsilon$. This shows $f(t; \theta)$ can be made
arbitrarily small by letting $n$ be large enough.
	
Similar to the case when $j = 0$, for $j \ge 1$,
$\E [ \vert \partial_\theta^{\otimes j} \lambda^{x,n}_{t, \theta} - \partial_\theta^{\otimes j} \lambda^{x, n, c}_{t,
  \theta} \vert ]$ is bounded by
\begin{align*}
  & \big \vert \partial_\theta^{\otimes j} \nu(x + \frac{t}{n}; \theta) - \partial_\theta^{\otimes j} \nu(x; \theta) \big\vert + \int_0^{nx-} \big\vert \partial_\theta^{\otimes j} g(t + nx - s; \theta) \big\vert \E \big[ \lambda^n_{s, \theta^0} \big] \rmd s \\
  + \ &  \E \Big[ \int_0^{t-} \big\vert \partial_\theta^{\otimes j} g(t-s; \theta) \big\vert \big\vert \lambda^{x,n}_{s, \theta^0} - \lambda^{x, n, c}_{s, \theta^0} \big\vert \rmd s\Big].
\end{align*}
By the assumed integrability of $\vert \partial_\theta^{\otimes j} g(\cdot; \theta) \vert$ and an adaptation of the
above proof for the case when $j = 0$, each term in the above expression can be made arbitrarily small by letting $n$ be
large enough.
	
Next, we extend the result for $p \ge 2$ by induction on $p$. Let us consider the case where $\theta = \theta^0$ first
again, $\E[ \vert \lambda^{x,n}_{t,\theta^0} - \lambda^{x, n, c}_{t, \theta^0} \vert^p]$ is bounded by
\begin{align*}
  &(1 + \epsilon)^{2(p-1)} \Big\{  \big\vert \nu(x + \frac{t}{n};\theta^0) - \nu(x ;\theta^0) \big\vert^p \\
  +{} &  \big( (1+ \epsilon) (1 + \epsilon^{-1})\big)^{p-1} \E \Big[ \Big( \int_0^{nx-} g(t + nx - s; \theta^0) \rmd N^n_s \Big)^p \Big] \\
  + {} & ( 1 + \epsilon^{-1})^{p-1} \E \Big[ \Big( \int_0^{t-} g(t-s; \theta^0) \big\vert \rmd N^{x,n}_s - \rmd N^{x,n,c}_s\big\vert \Big)^p \Big].
\end{align*}
Similar to before, $\vert \nu(x + \frac{t}{n}; \theta^0) - \nu(x; \theta^0) \vert^p$ can be made arbitrarily small by
increasing $n$. By Lemma A.2 of \cite{clinet2017statistical}, we have the following bound for
$\E [ ( \int_0^{nx-} g(t + nx - s; \theta^0) \rmd N^n_s )^p ]$:
\begin{align}
  K_p \Big\{ {} &\E \Big[ \int_0^{nx-} g(t + nx - s; \theta^0)^p  \lambda^n_{s, \theta^0}  \rmd s \Big] + \E \Big[ \Big( \int_0^{nx-} g(t + nx - s; \theta^0)^2 \lambda^n_{s, \theta^0} \rmd s \Big)^{p/2} \Big] \nonumber \\
  {} & +  \E \Big[ \Big( \int_0^{nx-} g(t + nx - s; \theta^0)\lambda^n_{s, \theta^0} \rmd s \Big)^p \Big]\Big\}. \label{eq:Chap_4_Step_2_eq1}
\end{align}
Similar to \eqref{eq:Chap_4_Holder_ineq_eq1}, by H\"{o}lder's inequality, for $i, k \ge 1$, we have
\begin{align*}
  & \E \Big[ \Big( \int_0^{nx-} g(t + nx - s; \theta^0)^i \lambda^n_{s, \theta^0} \rmd s \Big)^k \Big] \\
  \le {} & \Big( \int_0^{nx-} g(t + nx - s; \theta^0)^i  \rmd s \Big)^{k-1}  \E \Big[ \Big( \int_0^{nx-} g(t + nx - s; \theta^0)^i  \big( \lambda^n_{s, \theta^0} \big)^k \rmd s \Big) \Big]\\
  \le {} & K_k   \Big( \int_0^{nx-} g(t + nx - s; \theta^0)^i  \rmd s \Big)^k.
\end{align*}
By the integrability of $g(\cdot; \theta^0)^i$ for $i \le p$, we deduce the supremum of \eqref{eq:Chap_4_Step_2_eq1} in
$t$ over $[\gamma_n, b_n]$ converges to 0 as with $n$; hence so does
$\E [ ( \int_0^{nx} g(t + nx - s; \theta^0) \rmd N^n_s )^p ]$. By Poisson embedding and Lemma A.2 of
\cite{clinet2017statistical} again,
$ \E [ ( \int_0^{t-} g(t-s; \theta^0) \vert \rmd N^{x,n}_s - \rmd N^{x,n,c}_s\vert )^p ]$ is bounded by
\begin{align}
  K_p \Big\{ {} & \E \Big[ \int_0^{t-} g(t-s; \theta^0)^p \big\vert \lambda^{x,n}_{s, \theta^0} - \lambda^{x,n,c}_{s, \theta^0} \big\vert \Big] + \E \Big[ \Big( \int_0^{t-} g(t-s; \theta^0)^2 \big\vert \lambda^{x,n}_{s, \theta^0} - \lambda^{x,n,c}_{s, \theta^0} \big\vert \rmd s \Big)^{p/2}\Big] \Big\} \nonumber \\
		& + (1 + \epsilon^{-1})^{p-1} \E \Big[\Big( \int_0^{t-} g(t-s; \theta^0) \big\vert \lambda^{x,n}_{s, \theta^0} - \lambda^{x,n,c}_{s, \theta^0} \big\vert \rmd s \Big)^p\Big]. \label{eq:Chap_4_Step_2_eq2}
\end{align}
By H\"{o}lder's inequality again, we have
\begin{align*}
  & \E \Big[ \Big( \int_0^{t-} g(t - s; \theta^0)^i  \big\vert \lambda^{x,n}_{s, \theta^0} - \lambda^{x,n,c}_{s, \theta^0} \big\vert  \rmd s \Big)^k \Big] \\
  \le {} & \Big( \int_0^{t-} g(t - s; \theta^0)^i  \rmd s \Big)^{k-1}  \E \Big[ \Big( \int_0^{t-} g(t - s; \theta^0)^i  \big\vert \lambda^{x, n}_{s, \theta^0} - \lambda^{x,n,c}_{s, \theta^0} \big\vert^k \rmd s \Big) \Big].
\end{align*}
We can deduce from this inequality that, when taking the supremum in $t$ over $[\gamma_n, b_n]$, the first line of
\eqref{eq:Chap_4_Step_2_eq2} converges to 0 with $n$. To see this, note that there exists $\tilde{\gamma}_n \to\infty$
with $\tilde{\gamma}_n = o(\gamma_n)$ such that we can write
\begin{align}
  & \E \Big[  \int_0^{t-} g(t - s; \theta^0)^i  \big\vert \lambda^{x, n}_{s, \theta^0} - \lambda^{x}_{s, \theta^0} \big\vert^k \rmd s  \Big] \nonumber \\
  = {} & \int_0^{\tilde{\gamma}_n} g(t-s; \theta^0)^i \E \big[\big\vert \lambda^{x,n}_{s, \theta^0} - \lambda^{x,n,c}_{s, \theta^0} \big\vert^k \big] \rmd s + \int_{\tilde{\gamma}_n}^{t-} g(t-s; \theta^0)^i \E \big[\big\vert \lambda^{x,n}_{s, \theta^0} - \lambda^{x,n,c}_{s, \theta^0} \big\vert^k \big] \rmd s \nonumber \\
  \le {} & K_j  \int_0^{\tilde{\gamma}_n} g(t-s; \theta^0)^i \rmd s +   \sup_{t \in [\tilde{\gamma}_n, b_n]} \E \big[\big\vert \lambda^{x,n}_{t, \theta^0} - \lambda^{x,n,c}_{t, \theta^0} \big\vert^k \big]  \int_{\tilde{\gamma}_n}^{t-} g(t-s; \theta^0)^i \rmd s. \label{eq:Chap_4_Step_2_eq4}
\end{align}
By the induction hypothesis, we have
$\sup_{t \in [\tilde{\gamma}_n, b_n]} \E[\vert \lambda^{x,n}_{t, \theta^0} - \lambda^{x,n,c}_{t, \theta^0}\vert^k] =
o(1)$ for $k < p$. Thus, provided that $i \le p$ and $k < p$, along with the integrability of $g(\cdot; \theta^0)^i$,
the right hand side of \eqref{eq:Chap_4_Step_2_eq4} converges to $0$ with $n$. This shows the first line of
\eqref{eq:Chap_4_Step_2_eq2} converges to 0. Next, let
$f(t; \theta) = \E[ \vert \lambda^{x,n}_{t, \theta} - \lambda^{x,n,c}_{t, \theta} \vert^p]$. Thus far we have shown that
for any $\tilde{\epsilon}>0$, there exists $N$ large enough such that for all $n > N$,
\begin{align}
  f(t ;\theta^0)  \le {} & \tilde{\epsilon} + (1 + \epsilon^{-1})^{2(p-1)} \big(g (\cdot; \theta^0) \ast f(\cdot ;\theta^0) \big) (t) \label{eq:Chap_4_convol_eq1} \nonumber \\
  \le {} & \tilde{\epsilon}\Big( 1 + \sum_{i=1}^k (1 + \epsilon^{-1})^{2i(p-1)}G^{\ast (i)}(t; \theta^0) \Big) \nonumber\\
  {} & + (1 + \epsilon^{-1})^{2(k+1)(p-1)} \big( g^{\ast (k+1)} (\cdot; \theta^0) \ast f(\cdot; \theta^0) \big) (t) \nonumber \\
  \le {} & \tilde{\epsilon}\Big( 1 + \sum_{i=1}^k (1 + \epsilon^{-1})^{2i(p-1)}G^{\ast (i)}(t; \theta^0) \Big) + K_p \big( ( 1 + \epsilon^{-1})^{2(p-1)}\bar{\eta}\big)^{k+1}. \nonumber
\end{align}
Since $\bar{\eta} <1$, let $\epsilon$ be large enough such that $(1 + \epsilon^{-1})^{2(p-1)} \bar{\eta} < 1$, we have
\begin{align*}
  f(t; \theta^0) \le \frac{\tilde{\epsilon}}{1 - (1 + \epsilon^{-1})^{2(p-1)} \bar{\eta}} + K_p \big( ( 1 + \epsilon^{-1})^{2(p-1)} \bar{\eta}\big)^{k+1} 
\end{align*}
Hence there exists $K \in \N$ large enough such that for all $k > K$, $f(t; \theta^0) = o(1)$.
	
It remains to generalise the result for any $\theta \in \Theta$ and $j = \{0, 1,2\}$. Note that
$\E[\vert \partial_\theta^{\otimes j} \lambda^{x,n}_{t, \theta} - \partial_\theta^{\otimes j} \lambda^{x,n,c}_{t,
  \theta} \vert^p]$ is bounded by
\begin{align*}
  K_p \Big\{  & \big\vert \partial_\theta^{\otimes j} \nu(x + \frac{t}{n};\theta) - \partial_\theta^{\otimes j} \nu(x; \theta) \big\vert + \E \Big[ \Big( \int_0^{nx-} \big\vert \partial_\theta^{\otimes j} g(t + nx - s; \theta) \big\vert \rmd N^n_s \Big)^p \Big] \\
              & + \E \Big[ \Big( \int_0^{t-} \big\vert \partial_\theta^{\otimes j} g(t-s; \theta) \big\vert \big\vert \rmd N^{x,n}_s - \rmd N^{x,n,c}_s\big\vert \Big)^p \Big] \Big\}
\end{align*}
which, by following the arguments presented above, can be made arbitrarily small by letting $n$ be large enough.

\section{Technical lemmas and proofs of main results}\label{sec:ErgodicTheory-prep_Proofs}
\subsection*{Proof of Lemma~\ref{lm:ergodicity-lm1}}
Let us show
\begin{equation}
  \bigg\vert \frac{1}{b_n}\int_0^{b_n} \psi\big (Y^{x,n}_{t,\theta, \theta^\ast}  \big) - \psi\big(Y^{x,n,c}_{t,\theta, \theta^\ast}  \big)\rmd t \bigg\vert \overset{\mathbb{P}}{\rightarrow}0 \label{eq:asym_ergodicity_lm_1_eq1}
\end{equation}
first, then generalise the convergence to be uniform in $x$. Note that there exists a monotonically increasing sequence $\gamma_n >0$ such that $\gamma_n = o(b_n)$. By Lemma~\ref{lm:prep-proofs-lm1}, we can replace
the lower limit of the above integral by $\gamma_n$ such that it suffices to show $\frac{1}{b_n} \int_{\gamma_n}^{b_n} \psi( Y^{x,n}_{t, \theta, \theta^\ast} ) - \psi(Y^{x,n,c}_{t, \theta,\theta^\ast} )
\rmd t$ converges to 0 in, say, the $\mathbb{L}_1$-norm instead.  Furthermore, by the mean value theorem, Fubini's theorem and H\"{o}lder's inequality, there exists a random variable
$c^{x,n}_{t, \theta, \theta^\ast} \in E_1$ that lies on the line segment joining $Y^{x,n}_{t, \theta,\theta^\ast}$ and $Y^{x,n,c}_{t, \theta, \theta^\ast}$ such that $ \Vert \frac{1}{b_n}\int_{\gamma_n}^{b_n}
\psi (Y^{x,n}_{t,\theta, \theta^\ast} ) - \psi(Y^{x,n,c}_{t,\theta, \theta^\ast} ) \rmd t \Vert $ is bounded by
\begin{equation*}
  \big\Vert Y^{x,n}_{t, \theta,\theta^\ast} - Y^{x,n,c}_{t, \theta, \theta^\ast} \big\Vert_p \rmd t.
\end{equation*} Since $\vert \nabla \psi \vert \in C_{p-1}(E_1, \R)$, by Lemma~\ref{lm:prep-proofs-lm1} and \ref{lm:prep-proofs-lm2}, there exists an $\alpha \in
[0,1]$ such that
\begin{equation*} \big\Vert \nabla \psi \big(c^{x,n}_{t, \theta, \theta^\ast} \big) \big\Vert_{\frac{p}{p-1}} \le K_p \sup_{t \in [0, b_n]} \big \Vert \alpha Y^{x,n}_{t, \theta, \theta^\ast} + (1 - \alpha)
  Y^{x,n,c}_{t, \theta, \theta^\ast} \big\Vert_p^{p-1} < \infty
\end{equation*} and $\sup_{t \in [\gamma_n, b_n]} \Vert Y^{x,n}_{t, \theta, \theta^\ast} - Y^{x,n,c}_{t, \theta, \theta^\ast} \Vert_p = o (1)$. This shows \eqref{eq:asym_ergodicity_lm_1_eq1}.
	
To generalise the convergence to be uniform in $x$, it suffices to show
$$\sup_{t \in [\gamma_n, b_n]} \sup_{x \in [0,1]} \big \vert
Y^{x,n}_{t,\theta} - Y^{x,n,c}_{t,\theta} \big\vert \converginP 0$$ as $n \to \infty$. Let us look at showing
$\sup_{t \in [\gamma_n, b_n]} \sup_{x \in [0,1]} \vert \lambda^{x,n}_{t,\theta} - \lambda^{x,n,c}_{t,\theta} \vert
\converginP 0$ first. Note that $N^{x,n}$ can be written as the superposition of two point processes $N^{x,n,0}$ and
$N^{x,n,1}$, where $N^{x,n,0}$ counts only the offspring events of $N^n$ that occurred after $nx$ which are triggered by
events of $N^n$ that occurred prior to $nx$, and $N^{x,n,1}$ is a Hawkes process with zero history, time-varying
baseline intensity $\nu(x + \frac{\cdot}{n};\theta^\ast)$ and excitation kernel $g(\cdot; \theta^\ast)$. $N^{x,n,0}_t$
and $N^{x,n,1}_t$ have intensities $\lambda^{x,n,0}_{t,\theta^\ast}$ and $\lambda^{x,n,1}_{t, \theta^\ast}$,
respectively, where for any $\theta$, we have
\begin{align*}
  \lambda^{x,n,0}_{t,\theta} = {} & \int_0^{nx-} g(nx + t - s;\theta) \rmd N^n_s + \int_0^{t-} g(t-s;\theta) \rmd N^{x,n,0}_s \\
  \lambda^{x,n,1}_{t, \theta} = {} & \nu(x + \frac{t}{n}; \theta) + \int_0^{t-} g(t-s; \theta) \rmd N^{x,n,1}_s.
\end{align*}
Note that $\lambda^{x,n,0}_{t, \theta}$ evaluates the excitation effects generated by all the events of $N^{x,n}$ prior
to $nx$ and all of their subsequent offspring events at $t$ post $nx$. Hence $N^{x,n,0}$ contains only offspring events
and does not generate any immigrant events on its own. As the branching ratio is bounded by $\bar{\eta} < 1$ under
\ref{cond:C1}, we have $\lim_{n \to \infty} \sup_{x \in [0,1]} N^{x,n,0}_{b_n} \le K < \infty$ almost surely. From this,
we deduce that $\sup_{t \in [\gamma_n, b_n]} \sup_{x \in [0,1]} \lambda^{x,n,0}_{t, \theta} \to 0$ almost surely. Next,
we define two additional auxiliary point processes. Let $\dot{N}^{\bar{\nu}}$ be a stationary Hawkes process with
baseline intensity $\bar{\nu}$ and excitation kernel $g(\cdot; \theta^\ast)$. $\dot{N}^{\bar{\nu}}_t$ has intensity
$\dot{\lambda}^{\bar{\nu}}_{t, \theta^\ast}$ where for any $\theta$, we have
\begin{equation*} \dot{\lambda}^{\bar{\nu}}_{t, \theta} = \bar{\nu} + \int_{-\infty}^{t-} g(t-s;\theta) \rmd
  \dot{N}^{\bar{\nu}}_s.
\end{equation*}
Next, let $\bar{\nu}^n(x; \theta) = \sup_{x' \in [x, x+\frac{b_n}{n}]} \nu(x'; \theta)$ and denote $\tilde{N}^{x,n,c}$
to be a Hawkes process with constant baseline intensity $\bar{\nu}^n(x; \theta^\ast)$, excitation kernel
$g(\cdot; \theta^\ast)$ and pre-excitation effects generated by events of $\dot{N}^{\bar{\nu}}$ prior to $nx$. That is,
$\tilde{N}^{x,n,c}$ counts the offspring events of $\dot{N}^{\bar{\nu}}$ after $nx$ which are triggered by events of
$\dot{N}^{\bar{\nu}}$ that occurred prior to $nx$, as well as the events generated by a Hawkes process with zero
history, constant baseline intensity $\bar{\nu}^n(x; \theta^\ast)$ and excitation kernel $g(\cdot;
\theta^\ast)$. $\tilde{N}^{x,n,c}_t$ is thus the superposition of two point processes $\tilde{N}^{x,n,c,0}_t$ and
$\tilde{N}^{x,n,c,1}_t$, and have intensities $\tilde{\lambda}^{x,n,c,0}_{t, \theta^\ast}$ and
$\tilde{\lambda}^{x,n,c,1}_{t, \theta^\ast}$, respectively, where for any $\theta$, we have
	\begin{align*}
          \tilde{\lambda}^{x,n,c,0}_{t, \theta} = {} & \int_{-\infty}^{nx-} g(nx + t - s; \theta) \rmd \dot{N}^{\bar{\nu}}_s + \int_0^{t-} g(t-s; \theta)\rmd \tilde{N}^{x,n,c,0}_s, \\
          \tilde{\lambda}^{x,n,c,1}_{t, \theta} = {} & \bar{\nu}^n(x; \theta) +\int_0^{t-} g(t-s;\theta) \rmd \tilde{N}^{x,n,c,1}_s.
	\end{align*}
        Under this construction, $\lambda^{x,n}_{\cdot, \theta}$ and $\lambda^{x,n,c}_{\cdot, \theta}$ are dominated by
        $\tilde{\lambda}^{x,n,c}_{\cdot, \theta}$, $\lambda^{x,n,0}_{\cdot, \theta}$ by $\tilde{\lambda}^{x,n,c,0}_{\cdot, \theta}$, and
        $\lambda^{x,n,1}_{\cdot, \theta} $ by $\tilde{\lambda}^{x,n,c,1}_{\cdot, \theta}$ almost surely. We have
\begin{align}
  0 \le {} & \tilde{\lambda}^{x,n,c}_{t, \theta} - \lambda^{x,n}_{t, \theta} \nonumber \\
  = {} & \bar{\nu}^n(x; \theta) - \nu(x + \frac{t}{n}; \theta) + \tilde{\lambda}^{x,n,c,0}_{t, \theta} - \lambda^{x,n,0}_{t, \theta} + \int_0^{t-}
         g(t-s; \theta) \big( \rmd \tilde{N}^{x,n,c,1}_s - \rmd N^{x,n,1}_s\big). \label{eq:unifconvergence_eq1}
\end{align} Noting that $t \in [\gamma_n, b_n]$ and $\nu(\cdot; \theta)$ is Lipschitz equicontinuous, we have $\bar{\nu}^n(x; \theta) - \nu(x + \frac{t}{n}; \theta) \le K \frac{b_n}{n} \to 0$
with $n$. Furthermore, from the above discussion, we deduce that $\sup_{t \in [\gamma_n, b_n]} \sup_{x \in [0,1]} \tilde{\lambda}^{x,n,c,0}_{t, \theta}$ and $\sup_{t \in [\gamma_n, b_n]} \sup_{x \in
  [0,1]} \lambda^{x,n,0}_{t, \theta}$ both tend to $0$ almost surely with $n$. Finally, by Poisson embedding, $\tilde{N}^{x,n,c,1} - N^{x,n,1}$ is dominated by another Hawkes process with with zero
history, baseline intensity $K\frac{b_n}{n}$ (which is independent from $x$) and excitation kernel $g(\cdot; \theta^\ast)$. Since $\frac{b_n}{n} \to 0$, the last term of \eqref{eq:unifconvergence_eq1}
tends to $0$ uniformly in $x$, almost surely.
	
Thus far, we have shown that $\sup_{t \in [\gamma_n, b_n]} \sup_{x \in [0,1]} \vert \tilde{\lambda}^{x,n,c}_{t, \theta} - \lambda^{x,n}_{t,\theta} \vert \to 0$ almost
surely. Hence for $\sup_{t \in [\gamma_n, b_n]} \sup_{x \in [0,1]} \vert \lambda^{x,n}_{t,\theta} - \lambda^{x,n,c}_{t,\theta} \vert \converginP 0$, it suffices that
$$\sup_{t \in [\gamma_n, b_n]} \sup_{x \in [0,1]} \vert \tilde{\lambda}^{x,n,c}_{t, \theta} - \lambda^{x,n,c}_{t,\theta} \vert \to 0$$ 
almost surely as well. Note again that $\lambda^{x,n,c}_{t, \theta}$ is dominated by $\tilde{\lambda}^{x,n,c}_{t, \theta}$, we write
\begin{align}
  0 \le {} & \tilde{\lambda}^{x,n,c}_{t, \theta} - \lambda^{x,n,c}_{t, \theta} \nonumber \\
  = {} & \bar{\nu}^n(x; \theta) - \nu(x; \theta) + \tilde{\lambda}^{x,n,c,0}_{t, \theta} + \int_0^{t-} g(t-s; \theta) \big( \rmd \tilde{N}^{x,n,c,1}_s - \rmd
         N^{x,n,c}_s\big). \label{eq:unifconvergence_eq2}
\end{align}
By similar arguments as before, we can deduce that \eqref{eq:unifconvergence_eq2} converges to 0 uniformly in $x \in [0,1]$, $t \in [\gamma_n, b_n]$ almost surely. This
proves $\sup_{t \in [\gamma_n, b_n]} \sup_{x \in [0,1]} \vert \lambda^{x,n}_{t,\theta} - \lambda^{x,n,c}_{t,\theta} \vert \converginP 0$. From the Lipschitz
equicontinuity of $\partial_\theta^{\otimes j}\nu(\cdot; \theta)$ and integrability of $\vert \partial_\theta^{\otimes j} g(\cdot; \theta) \vert$ for $j \in \{1, 2\}$, it
is straightforward to see that similar arguments lead to the uniform estimate
\begin{equation*}
  \sup_{t \in [\gamma_n, b_n]} \sup_{x \in [0,1]} \big\vert Y^{x,n}_{t, \theta, \theta^\ast} - Y^{x,n,c}_{t, \theta, \theta^\ast} \big\vert \to 0 \quad \mbox{almost surely}.
\end{equation*}
\subsection*{Proof of Lemma~\ref{lm:ergodicity-lm2}}
This is a straightforward adaptation of the proof of Lemma~\ref{lm:ergodicity-lm1} and the details are omitted here.
\subsection*{Proof of Lemma~\ref{lm:ergodicity-lm3}}
Let us consider the stationary Hawkes process $\dot{N}^x$ on $\mathbb{R}$ with constant baseline intensity
$\nu(x; \theta^\ast)$ and excitation kernel $g(\cdot; \theta^\ast)$. Denote the natural filtration of $\dot{N}^x$ by
$\dot{\mathcal{F}}^x =(\dot{\mathcal{F}}^x_t)_{t \in \R}$ with
$\dot{\mathcal{F}}^x_t = \sigma \{ \dot{N}^x_s: -\infty < s \le t \}$ and the intensity of $\dot{N}^x_t$ by
$\dot{\lambda}^x_{t, \theta^\ast}$ where for any $\theta$, we have
\begin{equation*}
  \dot{\lambda}^x_{t, \theta} + \nu(x; \theta) = \int_{-\infty}^{t-} g(t-s; \theta) \rmd \dot{N}^x_s. 
\end{equation*}
Let
$\dot{Y}^x_{t, \theta, \theta^\ast} = (\dot{\lambda}^x_{t, \theta}, \dot{\lambda}^x_{t, \theta^\ast}, \partial_\theta \dot{\lambda}^x_{t, \theta},
\partial_\theta^{\otimes 2} \dot{\lambda}^x_{t, \theta})$ take values on $E_1$. Instead of showing \eqref{eq:asympergodicityeq4}, it is equivalent to show
\begin{equation}
  \sup_{x \in [0,1]}  \bigg\vert \frac{1}{T}\int_{0}^{T} 	\psi\big( \dot{Y}^x_{t,\theta, \theta^\ast} \big) \rmd t - \dot{\pi}\big(x, \theta, \theta^\ast ;\psi \big) \bigg\vert\overset{\mathbb{P}}{\rightarrow}0. \label{eq:uniform_ergodicity_eq1}
\end{equation}
By Theorem 1 and Proposition 1 of \cite{kwan2024ergodic}, $\dot{Y}^x_{t, \theta, \theta^\ast}$ is $C_p(E_1, \R)$-ergodic and
$\dot{\pi}(x, \theta, \theta^\ast; \psi) = \E [\psi(\dot{Y}^x_{t, \theta, \theta^\ast})]$. It thus remains to show the convergence is uniform in $x$. Since
$\psi$ is continuous on $E_1$ and $[0,1]$ is compact, it suffices to show that $\frac{1}{T}\int_0^T \psi (\dot{Y}^x_{t, \theta, \theta^\ast}) \rmd t$ is
equicontinuous in, say, the $\mathbb{L}_1$-norm. By Fubini's theorem, H\"{o}lder's inequality and the mean value theorem, there exists some random variable
$\dot{c}^{x,x'}_{t, \theta, \theta^\ast}$ that lies on the line segment joining $\dot{Y}^x_{t, \theta, \theta^\ast}$ and $\dot{Y}^{x'}_{t, \theta, \theta^\ast}$
such that we have
\begin{align*}
  {} & \Big\Vert \frac{1}{T} \int_0^T \psi(\dot{Y}^x_{t, \theta, \theta^\ast} ) - \psi(\dot{Y}^{x'}_{t, \theta, \theta^\ast}) \rmd t \Big\Vert\\
  \le {} & \frac{1}{T} \int_0^T \big\Vert \nabla \psi (\dot{c}^{x,x'}_{t, \theta, \theta^\ast}) \big\Vert_{\frac{p}{p-1}} \big\Vert \dot{Y}^x_{t, \theta, \theta^\ast} - \dot{Y}^{x'}_{t, \theta, \theta^\ast} \big\Vert_p \rmd t \\
  \le{} & \sup_{t \in \R} \big\Vert \nabla \psi (\dot{c}^{x,x'}_{t, \theta, \theta^\ast}) \big\Vert_{\frac{p}{p-1}} \sup_{t \in \R} \Vert \dot{Y}^x_{t, \theta, \theta^\ast} - \dot{Y}^{x'}_{t, \theta, \theta^\ast} \big\Vert_p.
\end{align*}
Under the assumed regularity conditions, a straightforward adaptation of the proofs of Lemmas~\ref{lm:prep-proofs-lm1} and \ref{lm:prep-proofs-lm2} shows that
$\sup_{t \in \R} \Vert \nabla \dot{c}^{x, x'}_{t, \theta, \theta^\ast} \Vert_{\frac{p}{p-1}}$ is bounded and
$\sup_{t \in \R} \Vert \dot{Y}^x_{t, \theta, \theta^\ast} - \dot{Y}^{x'}_{t, \theta, \theta^\ast} \Vert_p \le K_p \vert x - x' \vert$. Thus for any
$\tilde{\epsilon} >0$, by taking $x$ sufficiently close to $x'$,
\begin{equation*}
  \Big\Vert \frac{1}{T} \int_0^T \psi(\dot{Y}^x_{t, \theta, \theta^\ast} ) - \psi(\dot{Y}^{x'}_{t, \theta, \theta^\ast}) \rmd t \Big\Vert < \tilde{\epsilon},
\end{equation*}
the bound being independent of $T$. Such equicontinuity along with pointwise convergence ensures uniform convergence.
\subsection*{Proof of Lemma~\ref{lm:ergodicity-lm4}}
Since $\dot{\pi}(x, \theta, \theta^\ast; \psi) = \E [ \psi(\dot{Y}^x_{t, \theta, \theta^\ast})]$, thus by the mean value theorem and H\"{o}lder's inequality,
there exists $\dot{c}^{x, x'}_{t, \theta, \theta^\ast}$ that lies on the line segment joining $\dot{Y}^{x'}_{t, \theta, \theta^\ast}$ and
$\dot{Y}^x_{t, \theta, \theta^\ast}$ such that $\vert \dot{\pi}(x, \theta, \theta^\ast; \psi) - \dot{\pi}(x', \theta, \theta^\ast; \psi) \vert$ is bounded by
\begin{equation*}
  \big\Vert \nabla \psi (\dot{c}^{x, x'}_{t, \theta,  \theta^\ast} )\big\Vert_{\frac{p}{p-1}} \big\Vert \dot{Y}^x_{t, \theta, \theta^\ast} - \dot{Y}^{x'}_{t,
    \theta, \theta^\ast} \big\Vert_p.
\end{equation*}
From the proof of Lemma~\ref{lm:ergodicity-lm3}, it is straightforward to see that the above expression is bounded by $K_p\vert x - x'\vert$.

\subsection*{Proof of Proposition~\Ref{prop:uniform-convergence}}
  Since $\Theta$ is compact, thus for any $\delta >0$, there exist a finite number of open balls with radius $\delta$
  whose union covers $\Theta$. Let $\tilde{\theta}_1, \ldots, \tilde{\theta}_K$ be the centre of these balls and for any
  $\theta \in \Theta$, let $\tilde{\theta}_{i(\theta)}$ be the centre of the ball that contains $\theta$. For any
  $\tilde{\epsilon}>0$, we have the following bound for
  $\PP(\sup_{\theta \in \Theta} \vert \frac{1}{n} \psi(Y^n_{t, \theta, \theta^0}) \rmd t - \pi(\theta, \theta^0;
  \psi)\vert > \tilde{\epsilon})$:
  \begin{align}
    &\mathbb{P}\bigg( \sup_{\theta \in \Theta } \Big\vert \frac{1}{n} \int_{0}^n \psi\big(Y^n_{t,\theta,\theta^0}\big) - \psi\big(Y^n_{t,\tilde{\theta}_{i(\theta)}, \theta^0}\big) \rmd t \Big\vert > \frac{\tilde{\epsilon}}{3}\bigg) \nonumber\\
    +\ &\mathbb{P}\bigg( \sup_{\theta \in \Theta} \Big\vert \frac{1}{n} \int_{0}^n \psi\big(Y^n_{t,\tilde{\theta}_{i(\theta)}, \theta^0}\big)\rmd t - \pi \big( \tilde{\theta}_{i(\theta)}, \theta^0; \psi \big) \Big\vert > \frac{\tilde{\epsilon}}{3}\bigg) \nonumber\\
    +\ &\mathbb{P}\bigg( \sup_{\theta \in \Theta} \big\vert \pi \big( \tilde{\theta}_{i(\theta)}, \theta^0; \psi \big) - \pi \big( \theta, \theta^0; \psi \big)  \big\vert > \frac{\tilde{\epsilon}}{3}\bigg). \label{eq:approxprocedureeq8}
  \end{align}
  As $\Theta$ is compact, by the continuity conditions imposed on $\nu$ and $g$ under
  \hyperref[cond:C4]{[C4]($p$)}, for any $\tilde{\epsilon} >0$, there exists a $\delta >0$ such that for all
  $\theta, \theta' \in \Theta$, $\vert \theta - \theta' \vert < \delta$ implies that
  $\vert \frac{1}{n} \int_0^n \psi(Y^n_{t, \theta, \theta^0}) - \psi (Y^n_{t, \theta', \theta^0} ) \rmd t \vert \le
  \frac{\tilde{\epsilon}}{3}$ with probability tending to 1. Hence the first term of \eqref{eq:approxprocedureeq8}
  converges to 0. The second term of \eqref{eq:approxprocedureeq8} is bounded by
  \begin{equation*}
    \sum_{i = 1}^K \PP \bigg(\Big\vert \frac{1}{n} \int_{0}^{n}\psi \big(Y^n_{t, \tilde{\theta}_{i}, \theta^0}\big) \rmd t -  \pi \big(\tilde{\theta}_{i}, \theta^0; \psi \big)  \Big\vert > \frac{\tilde{\epsilon}}{3}\bigg)
  \end{equation*}
  which also converges to 0 by Theorem~\ref{thm:asympergodicity}. Next, we turn to showing
  $\pi(\theta, \theta^0; \psi)$ is continuous in $\theta$, which in turn would imply the last term of
  \eqref{eq:approxprocedureeq8} equals 0. Recall that
  $\pi(\theta, \theta^0; \psi) = \int_0^1 \dot{\pi}(x, \theta, \theta^0; \psi) \rmd x$ and by Lemma~\ref{lm:ergodicity-lm4},
  $\{\dot{\pi}(x, \theta, \theta^0; \psi)\}_{\theta \in \Theta, \theta^0 \in \Theta, \psi \in D_p(E_1, \R)}$ is
  equicontinuous in $x$. Hence it suffices to show
  $\{\dot{\pi}(x, \theta, \theta^0; \psi)\}_{x \in [0,1], \theta^0 \in \Theta, \psi \in D_p(E_1, \R)}$ is equicontinuous
  in $\theta$. Following the proof of Lemma~\ref{lm:ergodicity-lm4}, by the mean value theorem and
  H\"{o}lder's inequality, $\vert \dot{\pi}(x, \theta, \theta^0; \psi) - \dot{\pi}(x, \theta', \theta^0; \psi)\vert$ is
  bounded by
  \begin{equation*}
    \big\Vert \nabla \psi(\dot{c}^x_{t, \theta, \theta', \theta^0}) \big\Vert_{\frac{p}{p-1}} \big\Vert \dot{Y}^x_{t, \theta, \theta^0} - \dot{Y}^x_{t, \theta', \theta^0}\big\Vert_p
  \end{equation*}
  where $\dot{c}^x_{t, \theta, \theta', \theta^0}$ lies on the line segment joining $\dot{Y}^x_{t, \theta, \theta^0}$
  and $\dot{Y}^x_{t, \theta', \theta^0}$. By Lemma~\ref{lm:prep-proofs-lm1},
  $\Vert \nabla \psi(\dot{c}^x_{t, \theta, \theta', \theta^0}) \Vert_{\frac{p}{p-1}}$ is bounded. So it remains to show
  $\{ \partial_\theta^{\otimes j} \dot{\lambda}^x_{t, \theta}\}_{x \in [0,1], t \in \R, j \in \{0, 1, 2\}}$ is
  equicontinuous in $\theta$ in $\mathbb{L}_p$-norm. Let us consider the case where $p = 1$ and $j = 0$ first. Let
  $\theta$, $\theta' \in \Theta$, $\E [ \vert \dot{\lambda}^x_{t, \theta} - \dot{\lambda}^x_{t, \theta'} \vert]$ is
  bounded by
  \begin{equation*}
    \big\vert \nu(x; \theta) - \nu(x; \theta') \big\vert +   \int_{-\infty}^{t-} \big\vert g(t-s; \theta) - g(t-s; \theta') \big\vert \E \big[ \dot{\lambda}^x_{s, \theta^0}\big] \rmd s.
  \end{equation*}
  Under \ref{cond:C4}, $\{\nu(x;\theta)\}_{x \in [0,1]}$ is equicontinuous in $\theta$. Hence it remains to
  show the second term above can be made arbitrarily small by letting $\theta'$ be sufficiently close to $\theta$. By
  Lemma~\ref{lm:prep-proofs-lm1} and the mean value theorem, there exists $\theta^\ast \in \Theta$ that
  lies on the line segment joining $\theta$ and $\theta'$ such that
  \begin{equation*}
    \int_0^{t-} \big\vert g(t-s; \theta) - g(t-s; \theta') \big\vert \E \big[ \lambda^n_{s, \theta^0}\big] \rmd s \le \big\vert \theta - \theta' \big\vert K \int_0^{t-} \big\vert \partial_{\theta^\ast} g(t-s; \theta^\ast) \big\vert \rmd s.
  \end{equation*}
  Since $\vert \partial_\theta g(\cdot; \theta) \vert$ is integrable for any $\theta \in \Theta$, it follows that the
  above bound can be made arbitrarily small by letting $\theta$ and $\theta'$ to be sufficiently close to each
  other. This proves the case when $p = 1$ and $j = 0$. The proof to generalise the result for $p >1$ and
  $j \in \{1, 2\}$ is a straightforward adaptation of the proof above, Lemmas~\ref{lm:prep-proofs-lm1}
  and \ref{lm:prep-proofs-lm2}, and a consequence of the integrability of
  $\vert \partial_\theta^{\otimes j} g(\cdot; \theta) \vert^p$, $j \in \{1, 2, 3\}$.

\section{Conditions on non-exponential $g$}\label{sec:Examples-of-g}
The partial derivatives of the generalised Pareto kernel with respect to $\alpha$ and $\beta$ up to order 2 are
\begin{align*}
	\partial_\alpha \tilde{g}_p(t; \alpha, \beta) = {} & \bigg( \frac{\big( \beta + \alpha t \big) \ln \big( 1 + \frac{\alpha t}{\beta} \big) - \alpha \big(\alpha + 1\big) t}{\alpha^2 \big(\beta + \alpha t \big)} \bigg) \tilde{g}_p(t; \alpha, \beta); \\
	\partial_\beta \tilde{g}_p(t; \alpha, \beta) = {} & \Big( \frac{t - \beta}{\beta \big(\beta + \alpha t \big)} \Big) \tilde{g}_p(t; \alpha, \beta); \\
	\partial_\alpha^2 \tilde{g}_p(t; \alpha, \beta) = {} & \Bigg( \frac{\Big( \frac{1}{\alpha} + 1 \Big)t^2 }{\big( \beta + \alpha t \big)^2 } - \frac{2 \ln\big( 1 + \frac{\alpha t}{\beta} \big)}{\alpha^3} + \frac{2t}{\alpha^2 \big(\beta + \alpha t \big) }  \Bigg) \tilde{g}_p(t; \alpha, \beta) \\
	{} & + \Bigg( \frac{\ln\big( 1 + \frac{\alpha t }{\beta}\big)}{\alpha^2} - \frac{\big( \frac{1}{\alpha} + 1\big)t}{\beta + \alpha t} \Bigg)^2 \tilde{g}_p(t; \alpha, \beta);\\
	\partial_\beta^2 \tilde{g}_p(t; \alpha,\beta) = {} & \Big( \frac{2 \beta^2 - (\alpha - 1) t^2 - 4 \beta t}{\beta^2 (\beta + \alpha t)^2} \Big)\tilde{g}_p(t; \alpha, \beta);\\
	\partial_\alpha \partial_\beta \tilde{g}_p(t; \alpha, \beta) = {} & \Bigg( \frac{\big( t - \beta \big) \Big( \big(\beta + \alpha t \Big) \ln \big( 1 + \frac{\alpha t}{\beta}  \big) - \alpha \big( 2 \alpha + 1 \big) t \Big)}{\alpha^2 \beta \big(\beta + \alpha \beta \big)^2} \Bigg) \tilde{g}_p(t; \alpha, \beta).
\end{align*}
Since $\tilde{g}_p(t; \alpha, \beta)$ decays at a polynomial rate of $1 + 1/\alpha$, it is straightforward to see from here that
$\tilde{g}_p(t; \alpha, \beta)$ satisfies \hyperref[cond:C4]{[C3](p)} for $\alpha < 1/2$ and $\beta > 0$. Following a similar approach, we duduce that
\hyperref[cond:C4]{[C4]($p$)-(ii)} is satisfied for $\alpha < \frac{1}{3}$ and $\beta >0$.

Denote $\Gamma'(\alpha)$ and $\Gamma''(\alpha)$ to be the first two derivatives of the Gamma function with respect to $\alpha$ respectively, the
partial derivatives of the Gamma density with respect to $\alpha$ and $\beta$, up to order 2, are
\begin{align*}
	\partial_\alpha \tilde{g}_g(t; \alpha, \beta) = {} & \Big( \ln(t) - \ln(\beta) - \frac{\Gamma'(\alpha) }{\Gamma(\alpha)} \Big) \tilde{g}_g(t; \alpha, \beta);\\
	\partial_\beta\tilde{g}_g(t; \alpha, \beta) = {} & \Big(\frac{t}{\beta^2} - \frac{\alpha}{\beta} \Big) \tilde{g}_g(t; \alpha, \beta);\\
	\partial_\alpha^2 \tilde{g}_g(t; \alpha, \beta) = {} & \Big( \Big( \ln(t) - \ln(\beta) - \frac{\Gamma'(\alpha) }{\Gamma(\alpha)} \Big)^2 -  \frac{\Gamma''(\alpha)}{\Gamma(\alpha)} + \Big( \frac{\Gamma'(\alpha) }{\Gamma(\alpha)} \Big)^2 \Big) \tilde{g}_g(t; \alpha, \beta);\\
	\partial_\beta^2 \tilde{g}_g(t; \alpha, \beta) = {} & \Big( \frac{\alpha}{\beta^2} - \frac{2t}{\beta^3} + \Big(\frac{t}{\beta^2} - \frac{\alpha}{\beta}\Big)^2 \Big) \tilde{g}_g(t; \alpha, \beta);\\
	\partial_\alpha \partial_\beta \tilde{g}_g(t; \alpha, \beta) = {} &\Big( \Big( \ln(t) - \ln(\beta) - \frac{\Gamma'(\alpha) }{\Gamma(\alpha)} \Big) \Big(\frac{t}{\beta^2} - \frac{\alpha}{\beta}\Big) - \frac{1}{\beta} \Big) \tilde{g}_g(t; \alpha, \beta).
\end{align*}
Also termed the digamma function, the first logarithmic derivative of the Gamma function is denoted by $\psi(\alpha) = \frac{\Gamma'(\alpha)}{\Gamma(\alpha)}$. By Theorem 1 of \cite{Guo_2014}, we have the following bound for $\psi(\alpha)$:
\begin{equation*}
	\ln(\alpha + \frac{1}{2}) - \frac{1}{\alpha} < \psi(\alpha) < \ln(\alpha + e^{-\gamma}) - \frac{1}{\alpha} ,
\end{equation*}
where $\gamma = 0.577 \cdots$ stands for the Euler-Mascheroni's constant, and by Lemma 1.2 of
\cite{batir2005inequalities}, we have
$\psi'(\alpha) = \frac{\Gamma''(\alpha)}{\Gamma(\alpha)} - \frac{\Gamma'(\alpha)^2}{\Gamma(\alpha)^2} <
e^{-\psi(\alpha)}$.  From here, we deduce that $\tilde{g}_g(t; \alpha, \beta)$ satisfies \ref{cond:C3} and
\hyperref[cond:C4]{[C4]($p$)-(ii)} for $\alpha > 1-\frac{1}{p}$ and $\beta >0$.

Finally, the first and second order partial derivatives of the Weibull kernel with respect to $\alpha$ and $\beta$ are 
\begin{align*}
	\partial_\alpha \tilde{g}_w(t; \alpha, \beta) = {} & \Big( \frac{1}{\alpha} + \ln \Big( \frac{t}{\beta}\Big) \Big( 1 - \frac{t^\alpha}{\beta^\alpha} \Big) \Big) \tilde{g}_w(t; \alpha, \beta); \\
	\partial_\beta \tilde{g}_w(t; \alpha, \beta) = {} & \Big( \frac{\alpha t^\alpha}{\beta^{\alpha + 1}} - \frac{\alpha}{\beta} \Big) \tilde{g}_w(t; \alpha, \beta); \\
	\partial_\alpha^2 \tilde{g}_w(t; \alpha, \beta) = {} & \Big( \Big( \frac{1}{\alpha} + \ln \Big( \frac{t}{\beta}\Big) \Big( 1 - \frac{t^\alpha}{\beta^\alpha} \Big) \Big)^2 - \frac{1}{\alpha^2} - \frac{t^\alpha}{\beta^\alpha} \ln \Big(\frac{t}{\beta}\Big)^2 \Big) \tilde{g}_w(t; \alpha, \beta);\\
	\partial_\beta^2 \tilde{g}_w(t; \alpha, \beta) = {} & \Big( \frac{\alpha}{\beta^2} - \frac{\alpha (\alpha + 1) t^\alpha}{\beta^{\alpha + 2}}  + \Big( \frac{\alpha t^\alpha}{\beta^{\alpha + 1}} - \frac{\alpha}{\beta}\Big)^2 \Big) \tilde{g}_w(t; \alpha, \beta); \\
	\partial_\alpha \partial_\beta \tilde{g}_w(t; \alpha, \beta) = {} & \Big( \frac{t^\alpha }{\beta^{\alpha + 1}} \Big( 1 + \alpha \ln \Big(\frac{t}{\beta}\Big) \Big) - \frac{1}{\beta} \Big) \tilde{g}_w(t; \alpha, \beta)\\
	& + \frac{\alpha}{\beta} \Big(\frac{ t^\alpha}{\beta^\alpha} - 1  \Big) \Big( \frac{1}{\alpha} + \ln\Big( \frac{t}{\beta}\Big) \Big( 1 - \frac{t^\alpha}{\beta^\alpha} \Big)\Big) \tilde{g}_w(t; \alpha, \beta).
\end{align*}
Similar to the Gamma kernel, a quick calculation shows that $\tilde{g}_w(t; \alpha, \beta)$ satisfies \ref{cond:C3} and \hyperref[cond:C4]{[C4]($p$)-(ii)} for $\alpha > 1-\frac{1}{p}$ and $\beta >0$.

It remains to show that \hyperref[cond:C5]{[C5]-(iii)} holds. Note that for all three kernels considered in this
subsection, there exist $t$, $t' \in \R_+$ such that $\partial_{\alpha} \tilde{g}(t; \alpha, \beta) \ne 0$ and
$\partial_\beta \tilde{g}(t'; \alpha, \beta) \ne 0$. From this, we deduce that $\tilde{g}(t; \tilde{\theta}_g) =
\tilde{g}(t; \tilde{\theta}_g')$ for all $t \in\R_+$ only when $\tilde{\theta}_g = \tilde{\theta}_g'$. Next, note again
that for all three candidates of $g$, there exists a $t \in \R_+$ such that $\partial_{\theta_{g, i}} g(t; \theta_g) \ne
0$ for all $i \in \{1, \ldots, d_g\}$. This implies the latter of \hyperref[cond:C5]{[C5]-(iii)} holds.

\end{appendix}

\bibliographystyle{abbrv}
\bibliography{mybibfile}
\end{document}